\documentclass[11pt]{amsart}
\usepackage{amsmath}
\usepackage{amsfonts,mathrsfs}
\usepackage{amsthm}
\usepackage{mathtools}
\usepackage{amssymb}
\usepackage{verbatim}
\usepackage{enumerate}
\usepackage{graphicx}
\usepackage{longtable}
\usepackage[all,cmtip]{xy}
\usepackage{upref}
\usepackage[colorlinks=true,linkcolor=blue]{hyperref}
\usepackage[hang,small,bf]{caption}

\newtheorem{dfn}{Definition}[section]
\newtheorem{thm}[dfn]{Theorem}
\newtheorem{prp}[dfn]{Proposition}
\newtheorem{rmk}[dfn]{Remark}
\newtheorem{lem}[dfn]{Lemma}
\newtheorem{cor}[dfn]{Corollary}
\newtheorem{exe}{Exercise}

\newcommand{\T}{{\mathbb T}}
\newcommand{\Z}{{\mathbb Z}}
\newcommand{\R}{{\mathbb R}}
\newcommand{\N}{{\mathbb N}}
\newcommand{\C}{{\mathbb C}}
\newcommand{\op}[1]{\operatorname{#1}}
\newcommand{\di}{{\mathrm{d}}}

\begin{document}

\title[On closed orbits for twisted autonomous Tonelli Lagrangian flows]{Lecture notes on closed orbits \\ for twisted autonomous Tonelli Lagrangian flows\\}

\author{Gabriele Benedetti}
\address{WWU M\"unster, Mathematisches Institut, Einsteinstrasse 62, D-48149 M\"unster, Germany}
\email{\href{mailto:benedett@uni-muenster.de}{benedett@uni-muenster.de}}

\date{\today}

\begin{abstract}
These notes were prepared in occasion of a mini-course given by the author at the ``CIMPA Research School - Hamiltonian and Lagrangian Dynamics" (10--19 March 2015 - Salto, Uruguay). The talks were meant as an introduction to the problem of finding periodic orbits of prescribed energy for autonomous Tonelli Lagrangian systems on the twisted cotangent bundle of a closed manifold. In the first part of the lecture notes, we put together in a general theorem old and new results on the subject. In the second part, we focus on an important class of examples: magnetic flows on surfaces. For such systems, we discuss a special method, originally due to Ta\u\i manov, to find periodic orbits with low energy and we study in detail the stability properties of the energy levels.\end{abstract}

\maketitle

\tableofcontents
\section{Introduction}
The study of invariant sets plays a crucial role in the understanding of the properties of a dynamical system: it can be used to obtain information on the dynamics both at a local scale, for example to determine the existence of nearby stable motions, and at a global one, for example to detect the presence of chaos. In this regard we refer the reader to the monograph \cite{Mos73}. In the realm of continuous flows periodic orbits are the simplest example of invariant sets and, therefore, they represent the first object of study. For systems admitting a Lagrangian formulation closed orbits received special consideration in the past years, in particular for the cases having geometrical or physical significance, such as geodesic flows \cite{Kli78} or mechanical flows in phase space \cite{Koz85}. In \cite{Con06} Contreras formulated a very general theorem about the existence of periodic motions for autonomous Lagrangian systems over compact configuration spaces. Later on, this result was analysed in detail by Abbondandolo, who discussed it in a series of lecture notes \cite{Abb13}. It is the purpose of the present work to present a generalization of such a theorem, based on the recent papers \cite{Mer10,AB15b}, to systems which admit only a \textit{local} Lagrangian description (Theorem \ref{thm:main} below). Among these systems we find the important example of \textit{magnetic flows on surfaces}, which we introduce in Section \ref{sub:hom}. We look at them in detail in the last part of the notes: first, we will sketch a method, devised by Ta\u\i manov in \cite{Tai93}, to find periodic orbits with low energy; second, we will study the \textit{stability} of the energy levels, a purely symplectic property, which has important consequences for the existence of periodic orbits.

Let us start now our study by making precise the general setting in which we work. 

\subsection{Twisted Lagrangian flows over closed manifolds}\label{sub:twi}
Let $M$ be a closed connected $n$-dimensional manifold and denote by
\begin{equation*}
\begin{aligned}
\pi:TM&\ \longrightarrow\ M\\
(q,v)&\ \longmapsto\ q
\end{aligned}\quad\quad\quad\quad
\begin{aligned}
\pi:T^*M&\ \longrightarrow\ M\\
(q,p)&\ \longmapsto\ q
\end{aligned}
\end{equation*}
the tangent and the cotangent bundle projection of $M$. Let us fix also an auxiliary Riemannian metric $g$ on $M$ and let $|\cdot|$ denote the associated norm.

Let $\sigma\in\Omega^2(M)$ be a closed $2$-form on $M$ which we refer to as the \textit{magnetic form}. We call \textit{twisted cotangent bundle} the symplectic manifold $(T^*M,\omega_\sigma)$, where $\omega_\sigma:=d\lambda-\pi^*\sigma$. Here $\lambda$ is the canonical $1$-form defined by
\begin{equation*}
\lambda_{(q,p)}\ =\ p\circ d_{(q,p)}\pi\,, \quad\quad \forall\, (q,p)\in T^*M\,.
\end{equation*}
If $K:T^*M\rightarrow\R$ is a smooth function, we denote by $t\mapsto \Phi^{(K,\sigma)}_t$ the Hamiltonian flow of $K$. It is generated by the vector field $X_{(K,\sigma)}$ defined by
\begin{equation*}
\omega_\sigma(X_{(K,\sigma)},\,\cdot\,)\ =\ -dK\,.
\end{equation*}
In local coordinates on $T^*M$ such flow is obtained by integrating the equations
\begin{equation}
\left\{\begin{aligned}
       \dot q&=\ \frac{\partial K}{\partial p}\,,\\
       \dot p&=\ -\frac{\partial K}{\partial q}\ +\ \sigma\left(\frac{\partial K}{\partial p},\,\cdot\,\right)\,.
       \end{aligned}
\right.
\end{equation}
The function $K$ is an integral of motion for $\Phi^{(K,\sigma)}$. Moreover, if $k$ is a regular value for $K$, then the flow lines lying on $\{K=k\}$ are tangent to the $1$-dimensional distribution $\ker\omega_\sigma|_{\{K=k\}}$. This means that if $K':T^*M\rightarrow\R$ is another Hamiltonian with a regular value $k'$ such that $\{K'=k'\}=\{K=k\}$, then $\Phi^{(K',\sigma)}$ and $\Phi^{(K,\sigma)}$ are the same up to a \textit{time reparametrization} on the common hypersurface. In other words, there exists a smooth family of diffeomorphisms $\tau_{z}:\R\rightarrow\R$ parametrized by $z\in\{K'=k'\}=\{K=k\}$ such that
 \begin{equation*}
\tau_{z}(0)\ =\ 0\ \quad\mbox{and}\quad\ \Phi^{(K,\sigma)}_t(z)\ =\ \Phi^{(K',\sigma)}_{\tau_{z}(t)}(z)\,.
 \end{equation*}
Hence, there is a bijection between the closed orbits of the two flows on the hypersurface.

Let $L:TM\rightarrow\R$ be a \textit{Tonelli Lagrangian}. This means that for every $q\in M$, the restriction $L|_{T_qM}$ is superlinear and strictly convex (see \cite{Abb13}):
\begin{equation}\label{eq:ton}
\begin{aligned}
\lim_{|v|\rightarrow+\infty}\frac{L(q,v)}{|v|}\ &=\ +\infty\,,\quad\forall\, q\in M\,,\\
\frac{\partial^2L}{\partial v^2}(q,v)\ &>\ 0\,,\quad\forall\, (q,v)\in TM\,,
\end{aligned}
\end{equation}
where $\frac{\partial^2L}{\partial v^2}(q,v)$ is the Hessian of $L|_{T_qM}$ at $v\in T_qM$. The \textit{Legendre transform} associated to $L$ is the fibrewise diffeomorphism
\begin{align*}
\mathcal L:TM&\ \longrightarrow\ T^*M\\
(q,v)&\ \longmapsto\ \frac{\partial L}{\partial v}(q,v)\,.
\end{align*}
The \textit{Legendre dual} of $L$ is the \textit{Tonelli Hamiltonian}
\begin{align*}
H:T^*M&\ \longrightarrow\ \R\\
(q,p)&\ \longmapsto\ p\Big(\mathcal L^{-1}(q,p)\Big)-L\big(\mathcal L^{-1}(q,p)\big)\,,
\end{align*}
which satisfies the analogue of \eqref{eq:ton} on $T^*M$. For every $k\in\R$, let $\Sigma^*_k:=\{H=k\}$. These sets are compact and invariant for $\Phi^{(H,\sigma)}$. As a consequence such a flow is complete.
We can use $\mathcal L$ to pull back to $TM$ the Hamiltonian flow of $H$.

\begin{dfn}
Let $\Phi^{(L,\sigma)}$ be the flow on $TM$ defined by conjugation
\begin{equation*}
\mathcal L\circ\Phi^{(L,\sigma)}\ =\ \Phi^{(H,\sigma)}\circ\mathcal L\,.
\end{equation*}
We call $\Phi^{(L,\sigma)}$ a $\textsf{twisted\ Lagrangian\ flow}$ and we write $X_{(L,\sigma)}$ for its generating vector field. Since $\Phi^{(H,\sigma)}$ is complete, $\Phi^{(L,\sigma)}$ is complete as well.
\end{dfn}
The next proposition shows that the flow $\Phi^{(L,\sigma)}$ is locally a standard Lagrangian flow.
\begin{prp}
Let $U\subset M$ be an open set such that $\sigma|_U=d\theta$ for some $\theta\in\Omega^1(U)$. There holds
\begin{equation*}
X_{(L-\theta,0)}\ =\ X_{(L,\sigma)}|_U\,,
\end{equation*}
where $L-\theta:TU\rightarrow\R$ is the Tonelli Lagrangian defined by $(L-\theta)(q,v)=L(q,v)-\theta_q(v)$ and $X_{(L-\theta,0)}$ is the standard Lagrange vector field of $L-\theta$.
\end{prp}
\noindent The proof of this result follows from the next exercise.
\begin{exe}\label{exe-EL}
Prove the following generalization of the Euler-Lagrange equations. Consider a smooth curve $\gamma:[0,T]\rightarrow M$. Then, the curve $(\gamma,\dot\gamma)$ is a flow line of $X_{(L,\sigma)}$ if and only if for every open set $W\subset M$ and every linear symmetric connection $\nabla$ on $W$,
\begin{equation}
\left(\nabla_{\dot\gamma}\frac{\partial L}{\partial v}\right)(\gamma,\dot\gamma)\ =\ \frac{\partial L}{\partial q}(\gamma,\dot\gamma)\ +\ \sigma_{\gamma}(\dot\gamma,\cdot)
\end{equation}
at every time $t\in[0,T]$ such that $\gamma(t)\in W$. In the above formula $\frac{\partial L}{\partial q}\in T^*M$ denotes the restriction of the differential of $L$ to the horizontal distribution given by $\nabla$.
\end{exe}
\bigskip

\subsection{The magnetic form}
Let $[\sigma]\in H^2(M;\R)$ denote the cohomology class of $\sigma$. We observe that for any $\theta\in\Omega^1(M)$, there holds
\begin{equation*}
X_{(L+\theta,\sigma+d\theta)}=X_{(L,\sigma)}\,.
\end{equation*}
Since $L+\theta$ is still a Tonelli Lagrangian, we expect that general properties of the dynamics depend on $\sigma$ only via $[\sigma]$. Moreover, if $\theta\in\Omega^1(M)$ is defined by $\theta_q:=-\frac{\partial L}{\partial v}(q,0)$, then
\begin{equation*}
\min_{v\in T_qM}\Big(L(q,v)\ +\ \theta_q(v)\Big)\ =\ L(q,0)\ +\ \theta_q(0)\,,\quad\forall\,q\in M\,.
\end{equation*}
Therefore, without loss of generality we assume from now on that $L|_{T_qM}$ attains its minimum at $(q,0)$, for every $q\in M$.

We can refine the classification of $\sigma$ given by $[\sigma]$ by looking at the cohomological properties of its lift to the universal cover. Let $\widetilde \sigma$ be the pull-back of $\sigma$ to the universal cover $\widetilde M\rightarrow M$. We say that $\sigma$ is \textit{weakly exact} if $[\widetilde\sigma]=0$. This is equivalent to asking that
\begin{equation*}
\int_{S^2}u^*\sigma\ =\ 0\,,\quad\forall\,u:S^2\longrightarrow M\,.
\end{equation*}
We say that $\sigma$ admits a \textit{bounded weak primitive} if there is $\widetilde\theta\in\Omega^1(\widetilde M)$ such that $d\widetilde\theta=\widetilde\sigma$ and
\begin{equation*}
\sup_{\widetilde q\in\widetilde M}|\widetilde\theta_{\widetilde q}|\ <\ +\infty\,.
\end{equation*}
In this case we write $[\widetilde\sigma]_b=0$. Notice that both notions that we just introduced depend on $\sigma$ only via $[\sigma]$.
\begin{exe}\label{ex:sup}
If $M$ is a surface and $[\sigma]\neq0$, show that
\begin{itemize}
 \item if $M=S^2$, then $[\widetilde\sigma]\neq0$;
 \item if $M=\T^2$, then $[\widetilde\sigma]=0$, but $[\widetilde\sigma]_b\neq0$;
 \item if $M\notin\{S^2,\T^2\}$, then $[\widetilde\sigma]_b=0$.
\end{itemize}
Using the second point, prove that
\begin{itemize}
 \item if $M=\T^n$ and $[\sigma]\neq0$, then $[\widetilde\sigma]=0$, but $[\widetilde\sigma]_b\neq0$;
 \item if $M$ is any manifold and $[\widetilde\sigma]_b=0$, then
 \begin{equation*}
\int_{\T^2}u^*\sigma\ =\ 0\,,\quad\forall\,u:\T^2\longrightarrow M\,.
\end{equation*}
\end{itemize}
\end{exe}
\bigskip

\subsection{Energy}
As twisted Lagrangian flows are described by an autonomous Hamiltonian on the twisted cotangent bundle, they possess a natural first integral. It is the Tonelli function $E:TM\rightarrow\R$ given by $E:=H\circ\mathcal L$. We call it the \textit{energy} of the system and we write $\Sigma_k:=\{E=k\}$, for every $k\in\R$. Let $V:M\rightarrow\R$ denote the restriction of $E$ to the zero section and let $e_m(L)$ and $e_0(L)$ denote the minimum and maximum of $V$, respectively.

\begin{prp}
The energy can be written as
  \begin{equation*}
   E(q,v)\ =\ \frac{\partial L}{\partial v}(q,v)(v)\ -\ L(q,v)
  \end{equation*}
and, for every $q\in M$, we have
 \begin{equation*}
  \min_{v\in T_qM}E(q,v)\ =\ E(q,0)\ =\ V(q)\ =\ -L(q,0)\,.
 \end{equation*}
Moreover,
\begin{itemize}
 \item $k>e_0(L)$ if and only if $\pi:\Sigma_k\rightarrow M$ is an $S^{n-1}$-bundle (isomorphic to the unit tangent bundle of $M$).
 \item $k<e_m(L)$ if and only if $\Sigma_k=\emptyset$.
 \end{itemize}
\end{prp}
\begin{exe}
If $q_0\in M$ is a critical point of $V$, then $(q_0,0)$ is a constant periodic orbit of $\Phi^{(L,\sigma)}$ with energy $V(q_0)$.
\end{exe}
\bigskip

\subsection{The Ma\~n\'e critical value of the universal cover}
When $\sigma$ is weakly exact we define the \textit{Ma\~{n}\'{e} critical value} of the universal cover as
\begin{equation}\label{mandef}
c(L,\sigma)\ :=\ \inf_{d\widetilde\theta\,=\,\widetilde\sigma}\left(\,\sup_{\widetilde q\in\widetilde M}\widetilde H(\widetilde q,\widetilde\theta_{\widetilde q})\,\right)\ \in\ \R\cup\{+\infty\}\,,
\end{equation}
where $\widetilde H:T^*\widetilde M\rightarrow\R$ is the lift of $H$ to $\widetilde M$. This number plays an important role, since as it will be apparent from Theorem \ref{thm:main} and the examples in Section \ref{sub:hom} the dynamics on $\Sigma_k$ changes dramatically when $k$ crosses $c(L,\sigma)$.

\begin{prp}\label{prp-man}
If $\sigma$ is weakly exact, then
\begin{itemize}
 \item $c(L,\sigma)<+\infty$ if and only if $[\widetilde\sigma]_b=0$;
 \item $c(L,\sigma)\geq e_0(L)$;
 \item if $\sigma=d\theta_0$, where $\theta_0(\cdot)=\mathcal L(\cdot,0)$, then $c(L,\sigma)=e_0(L)$ and the converse is true, provided $e_0(L)=e_m(L)$;
 \item given two Tonelli Lagrangians $L_1$ and $L_2$ and two real numbers $k_1$ and $k_2$ such that $\{H_1=k_1\}=\{H_2=k_2\}$, then
 \begin{equation*}
  c(L_1,\sigma)\geq k_1\ \Longleftrightarrow\ c(L_2,\sigma)\geq k_2\ \ \mbox{and}\ \ c(L_1,\sigma)\leq k_1\ \Longleftrightarrow\ c(L_2,\sigma)\leq k_2\,.
 \end{equation*}
\end{itemize}
\end{prp}
\bigskip

\subsection{Example I: electromagnetic Lagrangians}
Let $g$ be a Riemannian metric on $M$ and $V:M\rightarrow\R$ be a function. Suppose that the Lagrangian is of \textit{mechanical type}, namely it has the form
\begin{equation*}
 L(q,v)\ =\ \frac{1}{2}|v|^2\ -\ V(q)\,,
\end{equation*}
where $|\cdot|$ is the norm associated to $g$. In this case we refer to $\Phi^{(L,\sigma)}$ as a \textit{magnetic flow} since we have the following physical interpretation of this system: it models the motion of a charged particle $\gamma$ moving in $M$ under the influence of a potential $V$ and a stationary magnetic field $\sigma$. Using Exercise \ref{exe-EL}, the equation of motion reads
\begin{equation}\label{El-sur}
 \nabla_{\dot{\gamma}}\dot\gamma\ =\ -\nabla V(\gamma)\ +\ Y_\gamma(\dot\gamma)\,,
\end{equation}
where $\nabla V$ is the gradient of $V$ and, for every $q\in M$, $Y_q:T_qM\rightarrow T_qM$ is defined by
\begin{equation*}
g_q(Y_q(v_1),v_2)\ =\ \sigma_q(v_1,v_2)\,,\quad\forall\, v_1,v_2\in T_qM\,.
\end{equation*}

\begin{exe}
Prove that, if $k>\max V$, $\Phi^{(L,\sigma)}|_{\Sigma_k}$ can be described in terms of a purely kinetic system. Namely, define the Jacobi metric $g_k:=\frac{k-V}{k}g$ and the Lagrangian $L_k(q,v):=\frac{1}{2}|v|^2_k$, where $|\cdot|_k$ is the norm induced by $g_k$. Using the Hamiltonian formulation, show that $\Phi^{(L,\sigma)}|_{\{E=k\}}$ is conjugated (up to time reparametrization) to $\Phi^{(L_k,\sigma)}|_{\{E_k=k\}}$, where $E_k$ is the energy function of $L_k$.
\end{exe}

In the particular case $M=S^2$, magnetic flows describe yet another interesting mechanical system. Consider a rigid body in $\R^3$ with a fixed point and moving under the influence of a potential $V$. Suppose that $V$ is invariant under rotations around the axis $\hat z$. We identify the rigid body as an element $\psi\in SO(3)$. Since $SO(3)$ is a Lie group, we use left multiplications to get $TSO(3)\simeq SO(3)\times\R^3\ni(\psi,\Omega)$, where $\Omega$ is the angular speed of the body. Thus, we have a Lagrangian system on $SO(3)$ with $L=\frac{1}{2}|\Omega|^2-V(\psi)$ and $\sigma=0$. Here $|\cdot|$ denote the metric induced by the tensor of inertia of the body.

The quotient of $SO(3)$ by the action of the group of rotations around $\hat z$ is a two-sphere. The quotient map $q:SO(3)\rightarrow S^2$ sends $\psi$ to the unit vector in $\R^3$, whose entries are the coordinates of $\hat z$ in the basis determined by $\psi$.

By the rotational symmetry, the quantity $\Omega\cdot \hat z$ is an integral of motion. Hence, for every $\omega\in\R$, the set $\{\Omega\cdot\hat z=\omega\}\subset TSO(3)$ is invariant under the flow and we have the commutative diagram
\begin{equation*}
\xymatrix{
 \big(\{\Omega\cdot\hat z=\omega\},X_{(L,0)}\big)\ar[r]^-{dq} \ar[d]_{\pi} & \big(TS^2,X_{(L_\omega,\sigma_\omega)}\big) \ar[d]^{\pi} \\
 SO(3)\ar[r]^{q} & S^2\,.}
\end{equation*}
The resulting twisted Lagrangian system $(L_\omega,\sigma_\omega)$ on $S^2$ can be described as follows:
\begin{itemize}
 \item $L_\omega(q,v)=\frac{1}{2}|v|^2-V_\omega(q)$, where $|\cdot|$ is the norm associated to a \textit{convex} metric $g$ on $S^2$ (independent of $\omega$) and $V_\omega$ is a potential (depending on $\omega$);
 \item $\sigma_\omega=\omega\cdot\kappa$, where $\kappa$ is the curvature form of $g$ (in particular $\sigma_\omega$ has integral $4\pi\omega$ and, if $\omega\neq0$, it is a symplectic form on $S^2$). 
\end{itemize}
The rigid body model presented in this subsection is described in detail in \cite{Kha79}. We refer the reader to \cite{Nov82}, for other relevant problems in classical mechanics that can be described in terms of twisted Lagrangian systems. 
\bigskip

\subsection{Example II: magnetic flows on surfaces}\label{sub:hom}
We now specialize further the example of electromagnetic Lagrangians that we discussed in the previous subsection and we consider purely kinetic systems on a closed oriented Riemannian surface $(M,g)$. In this case 
\begin{equation}
L(q,v)\,:=\ \frac{1}{2}|v|^2\,,
\end{equation}
and $\sigma=f\cdot\mu$, where $\mu$ is the metric area form and $f:M\rightarrow \R$. The magnetic endomorphism can be written as $Y=f\cdot\imath$, where $\imath:TM\rightarrow TM$ is the fibrewise rotation by $\pi/2$.
\begin{rmk}
If the surface is isometrically embedded in the Euclidean space $\R^3$, $Y$ is the classical Lorentz force. Namely, we have $Y_q(v)=v\times B(q)$, where $\times$ is the outer product of vectors in $\R^3$ and $B$ is the vector field $B:M\rightarrow\R^3$ perpendicular to $M$ and determined by the equation $\op{vol}_{\R^3}(B,\,\cdot,\,\cdot\,)=\sigma$, where $\op{vol}_{\R^3}$ is the Euclidean volume.
\end{rmk}
For purely kinetic systems $E=L$ and, therefore, the solutions of the twisted Euler-Lagrange equations are parametrized by a multiple of the arc length. More precisely, if $(\gamma,\dot\gamma)\subset\Sigma_k$, then $|\dot\gamma|=\sqrt{2k}$. In particular, the solutions with $k=0$ are exactly the constant curves. To characterise the solutions with $k>0$ we write down explicitly the twisted Euler-Lagrange equation \eqref{El-sur}:
\begin{equation}\label{El-sur2}
\nabla_{\dot\gamma}\dot\gamma=f(\gamma)\cdot\imath\dot\gamma\,.
\end{equation}
We see that $\gamma$ satisfies \eqref{El-sur2} if and only if $|\dot\gamma|=\sqrt{2k}$ and
\begin{equation}\label{cur-sur}
\kappa_\gamma=s\cdot f(\gamma)\,,\quad\quad s\,:=\ \frac{1}{\sqrt{2k}}\,,
\end{equation}
where $\kappa_\gamma$ is the geodesic curvature of $\gamma$. The advantage of working with Equation \eqref{cur-sur} is that it is invariant under orientation-preserving reparametrizations.

Let us do some explicit computations when the data are homogeneous. Thus, let $g$ be a metric of constant curvature on $M$ and let $\sigma=\mu$. When $M\neq\T^2$ we assume, furthermore, that the absolute value of the Gaussian curvature is $1$. 
By \eqref{cur-sur}, in order to find the trajectories of $\Phi^{(L,\sigma)}$ we need to solve the equation $\kappa_\gamma=s$ for all $s>0$.

Denote by $\widetilde M$ the universal cover of $M$. Then, $\widetilde{S^2}=S^2$, $\widetilde{\T^2}=\R^2$ and, if $M$ has genus larger than one, $\widetilde M=\mathbb H$, where $\mathbb H$ is the hyperbolic plane. Our strategy will be to study the trajectories of the lifted flow and then project them down to $M$. Working on the universal cover is easier since there the problem has a bigger symmetry group. Notice, indeed, that the lifted flow is invariant under the group of orientation preserving isometries $\op{Iso}_+(\widetilde M)$.

\subsubsection{The two-sphere}
Let us fix geodesic polar coordinates $(r,\varphi)\in(0,\pi)\times \R/2\pi\Z$ around a point $q\in S^2$ corresponding to $r=0$. The metric takes the form
$dr^2+(\sin r)^2d\varphi^2$. Let $C_{r}(q)$ be the boundary of the geodesic ball of radius $r$ oriented in the counter-clockwise sense. We compute $\kappa_{C_{r}(q)}=\frac{1}{\tan r}$. Observe that $\tan r$ takes every positive value exactly once for $r\in(0,\pi/2)$. Therefore, if $s>0$, the trajectories of the flow are all supported on $C_{r(s)}(q)$, where $q$ varies in $S^2$ and \begin{equation}
r(s)\ =\ \arctan\frac{1}{s}\,\in\, (0,\pi/2)\,.
\end{equation}
In particular, all orbits are closed and their period is
\begin{equation*}
T(s)\ =\ \frac{2\pi s}{\sqrt{s^2+1}}\,.
\end{equation*}

\subsubsection{The two-torus}
In this case we readily see that the trajectories of the lifted flow are circles of radius $r(s)=1/s$. In particular, all the orbits are closed and contractible. Their period is $T(s)=2\pi$, hence it is independent of $s$ (or $k$).

\subsubsection{The hyperbolic surface}
We fix geodesic polar coordinates $(r,\varphi)\in(0,+\infty)\times \R/2\pi\Z$ around a point $q\in \mathbb H$ corresponding to $r=0$. The metric takes the form $dr^2+(\sinh r)^2d\varphi^2$. Defining $C_r(q)$ as in the case of $S^2$, we find $\kappa_{C_r(q)}=\frac{1}{\tanh r}$. Observe that $\tanh r$ takes all the values in $(0,1)$ exactly once, for $r\in (0,+\infty)$. Therefore, if $s\in(1,+\infty)$, the trajectories of the flow are the closed curves $C_{r(s)}(q)$, where $q$ varies in $\mathbb H$ and
\begin{equation}
r(s)\ =\ \op{arc}\!\tanh\frac{1}{s}\,\in\, (0,+\infty)
\end{equation}
In particular, for $s$ in this range all periodic orbits are contractible. The formula for the periods now reads
\begin{equation*}
T(s)=\frac{2\pi s}{\sqrt{s^2-1}}\,.
\end{equation*}

To understand what happens, when $s\leq1$ we take the upper half-plane as a model for the hyperbolic plane.
Thus, let $\mathbb H=\{\,z=(x,y)\in\C\ |\ y>0\,\}$. In these coordinates, the hyperbolic metric has the form $\frac{dx^2+dy^2}{y^2}$ and
\begin{equation*}
\op{Iso}_+(\mathbb H)\ =\ \Big\{\,z\mapsto\frac{az+b}{cz+d}\ \,\Big|\ \,a,b,c,d\in\R\,, \ ad-bc\,=\,1\,\Big\}\,.
\end{equation*}
We readily see that the affine transformations $z\mapsto az$, with $a>0$ form a  subgroup of $\op{Iso}_+(\mathbb H)$. This subgroup preserves all the Euclidean rays from the origin and acts transitively on each of them. Hence, we conclude that such curves have constant geodesic curvatures. If $\varphi\in(0,\pi)$ is the angle made by such ray with the $x$-axis, we find that the geodesic curvature of such ray is $\cos\varphi$. In order to do such computation one has to write the metric using Euclidean polar coordinates centered at the origin. Using the whole isometry group, we see that all the segments of circle intersecting $\partial\mathbb H$ with angle $\varphi$ have geodesic curvature $\cos\varphi$.

We claim that if $s\in(0,1)$ and $\nu\neq0$ is a free homotopy class of loops of $M$, there is a unique closed curve $\gamma_{s,\nu}$ in the class $\nu$, which has geodesic curvature $s$.
The class $\nu$ correspond to a conjugacy class in $\pi_1(M)$. We identify $\pi_1(M)$ with the set of deck transformations and we let $F:\mathbb H\rightarrow\mathbb H$ be a deck transformation belonging to the given conjugacy class. By a standard result in hyperbolic geometry, $F$ has two fixed points on $\partial\mathbb H$ (remember, for example, that there exists a geodesic in $\mathbb H$ invariant under $F$). Then, $\gamma_{s,\nu}$ is the projection to $M$ of the unique segment of circle connecting the fixed points of $F$ and making an angle $\varphi=\arccos s$ with $\partial\mathbb H$. The uniqueness of $\gamma_{s,\nu}$ stems form the uniqueness of such segment of circle.

In a similar fashion, we consider the subgroup of $\op{Iso}_+(\mathbb H)$ made by the maps $z\mapsto z+b$, with $b\in\R$. It preserves the horizontal line $\{y=1\}$ and act transitively on it. Hence, such curve has constant geodesic curvature. A computation shows that it is equal to $1$, if it is oriented by $\partial_x$. Using the whole isometry group, we see that all the circles tangent to $\partial \mathbb H$ have geodesic curvature equal to $1$. Following \cite{Gin96} we see that there is no closed curve in $M$ with such geodesic curvature. By contradiction, if such curve exist, then its lift would be preserved by a non-constant deck transformation. We can assume without loss of generality that such lift is the line $\{y=1\}$. We readily see that the only elements in $\op{Iso}_+(\mathbb H)$ which preserve $\{y=1\}$ are the horizontal translation. However, no such transformation can be a deck transformation, since it has only one fixed point on $\partial\mathbb H$.

\begin{exe}
Show that in this case $c(L,\sigma)=\frac{1}{2}$.
\end{exe}
\bigskip

\subsection{The Main Theorem}
We are now ready to state the central result of this mini-course.
\begin{thm}\label{thm:main}
The following four statements hold.
\begin{enumerate}
 \item Suppose $[\widetilde\sigma]_b=0$. For every $k>c(L,\sigma)$,
 \begin{enumerate}
  \item there exists a closed orbit on $\Sigma_k$ in any non-trivial free homotopy class;
  \item if $\pi_{d+1}(M)\neq0$ for some $d\geq 1$, there exists a contractible orbit on $\Sigma_k$.
 \end{enumerate}
 \item Suppose $[\widetilde\sigma]=0$. There exists a contractible orbit on $\Sigma_k$, for almost every energy $k\in(e_0(L),c(L,\sigma))$.
 \item Suppose $[\widetilde\sigma]\neq0$. There exists a contractible orbit on $\Sigma_k$, for almost every energy $k\in(e_0(L),+\infty)$.
 \item There exists a contractible orbit on $\Sigma_k$, for almost every $k\in(e_m(L),e_0(L))$. 
\end{enumerate}
The set for which existence holds in (2), (3) and (4) contains all the $k's$ for which $\Sigma_k^*$ is a stable hypersurface in $(T^*M,\omega_\sigma)$ (see \cite[\textit{page 122}]{HZ94}). 
\end{thm}
In these notes, we will prove (1), (2) and (3) above by relating closed orbits of the flow to the zeros of a closed $1$-form $\eta_k$ on the space of loops on $M$. We introduce such form and prove some of its general properties in Section \ref{sec:act}. In Section \ref{sec:min} we describe an abstract minimax method that we apply in Section \ref{sec:geo} to obtain zeros of $\eta_k$ in the specific cases listed in the theorem. A proof of (4) relies on different methods and it can be found in \cite{AB15b}.

\begin{rmk}
When $[\sigma]=0$, the theorem was proven by Contreras \cite{Con06}. Point \textit{(1)} and \textit{(2)}, with the additional hypothesis $[\widetilde\sigma]_b=0$, were proven by Osuna \cite{Osu05}. Point \textit{(2)} was proven in \cite{Mer10,Mer16}, for electromagnetic Lagrangians, and in \cite{AB15b} for general systems. A sketch of the proof of point \textit{(3)} was given in \cite[Section 3]{Nov82} and in \cite[Section 3.2]{Koz85}. It was rigorously established in \cite{AB15b}. Point \textit{(4)} follows by employing tools in symplectic geometry. For the weakly exact case it can also be proven using a variational approach as shown in \cite[Section 7]{Abb13}. For Lagrangians of mechanical type and vanishing magnetic form the existence problem in such interval has historically received much attention (see \cite[Section 2]{Koz85} and references therein). 
\end{rmk}
\noindent We end up this introduction by defining the notion of stability mentioned in the theorem.
\bigskip

\subsection{Stable hypersurfaces}
In general, the dynamics on $\Sigma^*_k$ may exhibit very different behaviours as $k$ changes. However, given a regular energy level $\Sigma^*_{k_0}$, in some special cases we can find a new Hamiltonian $H':T^*M\rightarrow\R$ such that $\{H'=k'_0\}=\Sigma^*_{k_0}$ and such that $\Phi^{(H',\sigma)}|_{\{H'=k'_0\}}$ and $\Phi^{(H',\sigma)}|_{\{H'=k'\}}$ are conjugated, up to a time reparametrization, provided $k'$ is sufficiently close to $k'_0$.
\begin{dfn}
We say that an embedded hypersurface $\imath:\Sigma^*\longrightarrow T^*M$ is $\mathsf{stable}$ in the symplectic manifold $(T^*M,\omega_\sigma)$ if there exists an open neighbourhood $W$ of $\Sigma^*$ and a diffeomorphism $\Psi_W:\Sigma^*\times(-\varepsilon_0,\varepsilon_0)\rightarrow W$ with the property that:
\begin{itemize}
 \item $\Psi_W|_{\Sigma^*\times\{0\}}=\imath$;
 \item the function $H^W:W\rightarrow \R$ defined through the commutative diagram
\begin{equation*}
\xymatrix{
 \Sigma^*\times(-\varepsilon_0,\varepsilon_0)\ar[r]^-{\Psi_W} \ar[d]_{\op{pr}_2} & W \ar[dl]^-{H^W}\,, \\
 (-\varepsilon_0,\varepsilon_0) & }
\end{equation*}
is such that, for every $k\in(-\varepsilon_0,\varepsilon_0)$, $\Phi^{(H^W,\sigma)}|_{\{H^W=0\}}$ and $\Phi^{(H^W,\sigma)}|_{\{H^W=k\}}$ are conjugated by the diffeomorphism $w\mapsto\Psi_W(\imath^{-1}(w),k)$ up to time reparametrization. In this case, the reparametrizing maps $\tau_{(z,k)}$ vary smoothly with $(z,k)\in\Sigma^*\times(-\varepsilon_0,\varepsilon_0)$ and satisfy $\tau_{(z,0)}=\op{Id}_\R$, for all $z\in\Sigma^*$.
\end{itemize}
This implies that there is a bijection between the periodic orbits on $\Sigma^*=\{H^W=0\}$ and those on $\{H^W=k\}$.
\end{dfn}

Thanks to a result of Macarini and G.\ Paternain \cite{MP10}, if $\Sigma^*$ is the energy level of some Tonelli Hamiltonian, the function $H^W$ can be taken to be Tonelli as well.
\begin{prp}
Suppose that for some $k>e_0(L)$, $\Sigma^*_k$ is stable with stabilizing neighbourhood $W$. Up to shrinking $W$, there exists a Tonelli Hamiltonian $H_k:T^*M\rightarrow\R$ such that $H^W=H_k$ on $W$.
\end{prp}
In order to check whether an energy level is stable or not, we give the following necessary and sufficient criterion that can be found in \cite[Lemma 2.3]{CM05}.
\begin{prp}
A hypersurface $\Sigma_k^*$ is stable if and only if there exists $\alpha\in\Omega^1(\Sigma_k^*)$ such that
\begin{equation*}
\textit{(a)}\ \ d\alpha(X_{(H,\sigma)},\,\cdot\,)\ =\ 0\,,\quad\quad \textit{(b)}\ \ \alpha(X_{(H,\sigma)})(z)\ \neq\ 0\,,\quad\forall\, z\in\Sigma_k^*\,.
\end{equation*}
In this case $\alpha$ is called a $\mathsf{stabilizing\ form}$. The first condition is implied by the following stronger assumption
\begin{equation*}
\textit{(a')}\ \ d\alpha\ =\ \omega_\sigma|_{\Sigma^*_k}\,.
\end{equation*}
If (a') and (b) are satisfied we say that $\Sigma^*_k$ is of $\mathsf{contact\ type}$ and we call $\alpha$ a contact form. We distinguish between $\mathsf{positive}$ and $\mathsf{negative}$ contact forms according to the sign of the function $\alpha(X_{(H,\sigma)})$.
\end{prp}
\noindent In Section \ref{sec:sta}, we give some sufficient criteria for stability for magnetic flows on surfaces. 

\section{The free period action form}\label{sec:act}
For the proof of the Main Theorem we need to characterize the periodic orbits on $\Sigma_k$ via a variational principle on a space of loops. To this purpose we have first to adjust $L$.
\subsection{Adapting the Lagrangian}
Let us introduce a subclass of Tonelli Lagrangians whose fibrewise growth is quadratic. This will enable us to define the action functional on the space of loops with square-integrable velocity.
\begin{dfn}
We say that $L$ is $\mathsf{quadratic\ at\ infinity}$ if there exists a metric $g_\infty$ and a potential $V_\infty:M\rightarrow\R$ such that $L(q,v)=\frac{1}{2}|v|_\infty^2-V_\infty(q)$ outside a compact set.
\end{dfn}
The next result tells us that, if we look at the dynamics on a fixed energy level, it is not restrictive to assume that the Lagrangian is quadratic at infinity.
\begin{prp}
For any fixed $k\in\R$, there exists a Tonelli Lagrangian $L_k:TM\rightarrow\R$ which is quadratic at infinity and such that $L_k=L$ on $\{E\leq k_0\}$, for some $k_0>k$. By choosing $k_0$ sufficiently large, we can obtain $e_0(L)=e_0(L_k)$ and, if $[\widetilde\sigma]=0$, also $c(L,\sigma)=c(L_k,\sigma)$.
\end{prp}
From now on, we assume that $L$ is quadratic at infinity. In this case there exist positive constants $C_0$ and $C_1$ such that
\begin{equation}\label{est-quad}
C_1|v|^2\ -\ C_0\ \leq\ L(q,v)\ \leq\ C_1|v|^2\ +\ C_0\,,\quad\forall\,(q,v)\in TM\,.
\end{equation}
An analogous statement holds for the energy. 
\bigskip

\subsection{The space of loops}
We define the space of loops where the variational principle will be defined. Given $T>0$, we set
\begin{equation*}
 W^{1,2}(\R/T\Z,M)\ :=\ \Big\{\,\gamma:\R/T\Z\rightarrow M\ \Big|\ \gamma \mbox{ is absolutely continuous\,, }\int_0^T|\dot\gamma|^2\,\di t<\infty\, \Big\}\,.
\end{equation*}
Since we look for periodic orbits of arbitrary period, we want to let $T$ vary among all the positive real numbers $\R^+$. This is the same as fixing the parametrization space to $\T:=\R/\Z$ and keeping track of the period as an additional variable. Namely, we have the identification
\begin{align*}
\bigsqcup_{T>0}W^{1,2}(\R/T\Z,M)&\ \longrightarrow\ \Lambda\,:=\,W^{1,2}(\T,M)\times\R^+\\
\gamma(t)&\ \longmapsto\ \big(x(s):=\gamma(sT),T\big)\,.
\end{align*}
Given a free homotopy class $\nu\in[\T,M]$, we denote by $W^{1,2}_\nu\subset W^{1,2}(\T,M)$ and $\Lambda_\nu\subset\Lambda$ the loops belonging to such class. We use the symbol $0$ for the class of contractible loops.
\begin{prp}
The set $\Lambda$ is a Hilbert manifold with $T_{(x,T)}\Lambda\simeq T_xW^{1,2}\times\R$, where $T_xW^{1,2}\simeq W^{1,2}(\T,x^*(TM))$ is the space of absolutely continuous vector fields along $x$ with square-integrable covariant derivative. The metric on $\Lambda$ is given by $g_\Lambda=g_{W^{1,2}}+\di T^2$, where
\begin{equation*}
(g_{W^{1,2}})_x(\xi_1,\xi_2)\ :=\ \int_0^1g_{x(s)}(\xi_1(s),\xi_2(s))\,\di s\ +\ \int_0^1g_{x(s)}(\xi_1'(s),\xi_2'(s))\,\di s\,.
\end{equation*}
For any $T_->0$, $W^{1,2}\times[T_-,+\infty)\subset\Lambda$ is a complete metric space. 
\end{prp}
\noindent For more details on the space of loops we refer to \cite[Section 2]{Abb13} and \cite{Kli78}. We end this subsection with two more definitions, which will be useful later on. First, we let
\begin{equation*}
\frac{\partial}{\partial T}\ \in\ \Gamma(\Lambda)
\end{equation*}
denote the coordinate vector associated with the variable $T$. Then, if $x\in W^{1,2}$, we let
\begin{equation*}
e(x)\ :=\ \int^1_0|x'|^2\di s\quad\quad\mbox{and}\quad\quad \ell(x)\ :=\ \int^1_0|x'|\,\di s
\end{equation*}
be the $L^2$-\textit{energy} and the \textit{length} of $x$, respectively. We define analogous quantities for $\gamma\in\Lambda$. We readily see that $\ell(x)=\ell(\gamma)$ and $e(x)=Te(\gamma)$. Moreover, $\ell(x)^2\leq e(x)$ holds.
\bigskip

\subsection{The action form}
In this subsection, for every $k\in\R$, we construct $\eta_k\in\Omega^1(\Lambda)$, which vanishes exactly at the set of periodic orbits on $\Sigma_k$. Such $1$-form will be made of two pieces: one depending only on $L$ and $k$ and one depending only on $\sigma$. The first piece will be the differential of the function
\begin{align*}
A_k:\Lambda&\longrightarrow \R\\
\gamma&\longmapsto \int_0^T\Big[L(\gamma,\dot\gamma)+k\Big]\,\di t\ =\ T\cdot\int_0^1\left[L\left(x,\frac{x'}{T}\right)+k\right]\di s\,.
\end{align*}
Such function is well-defined since $L$ is quadratic at infinity (see \eqref{est-quad}). It was proven in \cite{AS09} that $A_k$ is a $C^{1,1}$ function (namely, $A_k$ is differentiable and its differential is locally uniformly Lipschitz-continuous).

In order to define the part of $\eta_k$ depending on $\sigma$, we first introduce a differential form $\tau^\sigma\in\Omega^1(W^{1,2})$ called the \textit{transgression} of $\sigma$. It is given by
\begin{equation*}
\tau^\sigma_x(\xi)\ :=\ \int_0^1\sigma_{x(s)}(\xi(s),x'(s))\,\di s\,,\quad\forall\,(x,\xi)\in TW^{1,2}\,.
\end{equation*}
By writing $\tau^\sigma$ in local coordinates, it follows that it is locally uniformly Lipschitz.

If $u:[0,1]\rightarrow W^{1,2}$ is a path of class $C^1$, then
\begin{equation}
\int_0^1u^*\tau^\sigma\ =\ \int_{[0,1]\times\T}\hat u^*\sigma\,,
\end{equation}
where $\hat u:[0,1]\times\T\rightarrow M$ is the cylinder given by $\hat u(r,t)=u(r)(t)$. If $u_a:\T\rightarrow W^{1,2}$ is a homotopy of closed paths with parameter $a\in[0,1]$, then we get a corresponding homotopy of tori $\hat u_a$. Since $\sigma$ is closed, the integral of $\hat u_a^*\sigma$ on $\T^2$ is independent of $a$. We conclude that the integral of $\tau^\sigma$ on $u_a$ does not depend on $a$ either. Namely, $\tau^\sigma$ is a \textit{closed form}.

\begin{dfn}
The $\mathsf{free\ period\ action\ form}$ at energy $k$ is $\eta_k\in\Omega^1(\Lambda)$ defined as
\begin{equation}
\eta_k\,:=\ dA_k\ -\ \op{pr}^*_{W^{1,2}}\tau^\sigma\,,
\end{equation}
where $\op{pr}^*_{W^{1,2}}:\Lambda\rightarrow W^{1,2}$ is the natural projection $(x,T)\mapsto x$.
\end{dfn}
\begin{prp}
The free period action form is closed and its zeros correspond to the periodic orbits of $\Phi^{(L,\sigma)}$ on $\Sigma_k$.
\end{prp}
The correspondence with periodic orbits follows by computing $\eta_k$ explicitly on $TW^{1,2}\times 0$ and on $\frac{\partial}{\partial T}$. If $\xi\in TW^{1,2}$, then
\begin{equation}
(\eta_k)_\gamma(\xi,0)\ =\ \int_0^T\Big[\frac{\partial L}{\partial q}(\gamma,\dot\gamma)\cdot\xi_T\, +\,\frac{\partial L}{\partial v}(\gamma,\dot\gamma)\cdot\dot\xi_T\,+\,\sigma_\gamma(\dot\gamma,\xi_T)\Big]\di t\,,
\end{equation}
where $\xi_T$ is the reparametrization of $\xi$ on $\R/T\Z$. In the direction of the period we have
\begin{align}
\nonumber (\eta_k)_\gamma\left(\frac{\partial}{\partial T}\right)\ =\ d_\gamma A_k\left(\frac{\partial}{\partial T}\right)&\ =\ \int_0^1L\left(x,\frac{x'}{T}\right)\di s\ +\ k\ -\ T\cdot\int_0^1\frac{\partial L}{\partial v}\left(x,\frac{x'}{T}\right)\cdot \frac{x'}{T^2}\,\di s\\
&\ =\ k-\int_0^1E\left(x,\frac{x'}{T}\right)\di s\label{etakper}\\
\nonumber&\ =\ k-\frac{1}{T}\int_0^TE(\gamma,\dot\gamma)\,\di t\,.
\end{align}
\bigskip

\subsection{Vanishing sequences}
Our strategy to prove existence of periodic orbits will be to construct zeros of $\eta_k$ by approximation.
\begin{dfn}
Let $\nu\in[\T,M]$ be a free homotopy class. A sequence $(\gamma_m)\subset\Lambda_\nu$ is called a $\mathsf{vanishing\ sequence}$ (at level $k$), if
 \begin{equation*}
  \lim_{m\rightarrow\infty}\left|\eta_k\right|_{\gamma_m}\ =\ 0\,.
 \end{equation*}
\end{dfn}
\noindent A limit point of a vanishing sequence is a zero of $\eta_k$. Thus, the crucial question is: when does a vanishing sequence admit limit points? Clearly, if $T_m\rightarrow 0$ or $T_m\rightarrow+\infty$ the set of limit points is empty. We now see that the opposite implication also holds.
\begin{lem}\label{lem:van-bou}
If $(\gamma_m)$ is a vanishing sequence, there exists $C>0$ such that
\begin{equation}\label{eq-et}
e(x_m)\ \leq\ C\cdot T_m^2\,.
\end{equation}
\end{lem}
\begin{proof}
We compute 
\begin{equation*}
C_1\cdot\frac{e(x_m)}{T_m^2}-C_0\ \stackrel{\mbox{}^{(\star)}}{\leq}\ \int_0^1E\left(x_m,\frac{x_m'}{T_m}\right)\di s\ =\ k-\eta^k_{\gamma_m}\left(\frac{\partial}{\partial T}\right)\ \stackrel{\mbox{}^{(\star\star)}}{\leq}\ k+\sup_m|\eta_k|_{\gamma_m}\,.
\end{equation*}
where in $(\star)$ we used \eqref{est-quad} applied to $E$, and in $(\star\star)$ we used that
\begin{equation*}
\left|\frac{\partial}{\partial T}\right|\ =\ 1\,.
\end{equation*}
The desired estimate follows by observing that, since the sequence $\big(\,|\eta_k|_{\gamma_m}\big)\subset[0,+\infty)$ is infinitesimal, it is also bounded from above. 
\end{proof}
\begin{prp}\label{prp-conv}
If $(\gamma_m)$ is a vanishing sequence and $0<T_-\leq T_m\leq T_+<+\infty$ for some $T_-$ and $T_+$, then $(\gamma_m)$ has a limit point.
\end{prp}
\begin{proof}
By compactness of $[T_-,T_+]$, up to subsequences, $T_m\rightarrow T_\infty>0$. By \eqref{eq-et}, the $L^2$-energy of $x_m$ is uniformly bounded. Thus, $(x_m)$ is uniformly $1/2$-H\"older continuous. By the Arzel\`a-Ascoli theorem, up to subsequences, $(x_m)$ converges uniformly to a continuous $x_\infty:\T\rightarrow M$. Therefore, $x_m$ eventually belongs to a local chart $\mathcal U$ of $W^{1,2}$. In $\mathcal U$, $\eta_k$ can be written as the differential of a standard action functional depending on time (see \cite{AB15b}) and the same argument contained in \cite[Lemma 5.3]{Abb13} when $\sigma=0$ implies that $(\gamma_m)$ has a limit point.
\end{proof}

In order to construct vanishing sequences we will exploit some geometric properties of $\eta_k$. One of the main ingredients to achieve this goal will be to define a vector field on $\Lambda$ generalizing the negative gradient vector field of the function $A_k$. We introduce it in the next subsection.

\subsection{The flow of steepest descent}
Let $X_k$ denote the vector field on $\Lambda$ defined by
\begin{equation*}
X_k\,:=\ -\,\frac{\sharp\,\eta_k}{\sqrt{1+|\eta_k|^2}}\,\,
\end{equation*}
where $\sharp$ denote the duality between $1$-forms and vector fields induced by $g_{\Lambda}$. Since $X_k$ is locally uniformly Lipschitz, it gives rise to a flow which we denote by $r\mapsto\Phi^k_r$. For every $\gamma\in\Lambda$, we denote by $u_\gamma:[0,R_\gamma)\rightarrow\Lambda$ the maximal positive flow line starting at $\gamma$. We say that $\Phi^k$ is \textit{positively complete} on a subset $Y\subset \Lambda$ if, for all $\gamma\in\Lambda$, either $R_\gamma=+\infty$ or there exists $R_{\gamma,Y}\in[0,R_\gamma)$ such that $u_\gamma(R_{\gamma,Y})\notin Y$.

Except for the scaling factor $1/\sqrt{1+|\eta_k|^2}$, the vector field $X_k$ is the natural generalization of $-\nabla A_k=-\sharp(dA_k)$ to the case of non-vanishing magnetic form. We introduce such scaling so that $|X_k|\leq 1$ and we can give the following characterization of the flow lines $u_\gamma$ with $R_\gamma<+\infty$.

\begin{prp}
Let $u:[0,R)\rightarrow\Lambda$ be a maximal positive flow line of $X_k$ and for all $r\in[0,R)$ set $u(r):= \gamma(r)=(x(r),T(r))$. If $R<+\infty$, then there exists a sequence $(r_m)_{m\in\N}\subset[0,R)$ and a constant $C$ such that
\begin{equation}
\lim_{m\rightarrow\infty}r_m=R\,,\quad\ \ \lim_{m\rightarrow\infty}T(r_m)=0\,,\quad\ \ e(x(r_m))\ \leq\ C\cdot T(r_m)^2\,,\ \ \forall\,m\in\N\,.
\end{equation}
\end{prp}
\begin{proof}
By contradiction, we suppose that $0<T_-:=\inf_{[0,R)} T(r)$. Since $|X_k|\leq1$, $u_\gamma$ is uniformly continuous and, by the completeness of $W^{1,2}\times[T_-,+\infty)$, there exists
\begin{equation*}
\gamma_\infty\,:=\ \lim_{r\rightarrow R}u(r)\,.
\end{equation*}
By the existence theorem of solutions of ODE's, there exists a neighbourhood $\mathcal B$ of $\gamma_\infty$ and $R_\mathcal B>0$ such that
\begin{equation*}
\forall\,\gamma\in \mathcal B\,,\quad r\longmapsto\Phi^{k}_r(\gamma)\, \mbox{ exists in }[0,R_\mathcal B]\,.
\end{equation*}
This contradicts the fact that $R$ is finite as soon as $r\in[0,R)$ is such that $\gamma(r)\in \mathcal B$ and $R-r<R_\mathcal B$. Therefore, $\inf T=0$. Hence, we find a sequence $r_m\rightarrow R$ such that $T(r_m)\rightarrow 0$ and, for every $m$, $\frac{dT}{dr}(r_m)\leq 0$. The last property implies that
\begin{equation}
0\ \geq\ \frac{dT}{dr}(r_m)\ =\ d_{u(r_m)}T(X_k)\ =\ -\frac{\eta_k\left(\frac{\partial}{\partial T}\right)}{\sqrt{1+|\eta_k|^2}}(u(r_m))\,.
\end{equation}
Finally, using Equation \eqref{etakper} and the estimates in \eqref{est-quad}, we have
\begin{equation*}
0\leq(\eta_k)_{u(r(m))}\left(\frac{\partial}{\partial T}\right)=k-\int_0^1E\left(x(r_m),\frac{x'(r_m)}{T(r_m)}\right)\di s\leq k-C_1\int_0^1\frac{|x'(r_m)|^2}{T(r_m)^2}\di s+C_0\,.\qedhere
\end{equation*}
\end{proof}

The above proposition shows that flow lines whose interval of definition is finite come closer and closer to the subset of constant loops. As we saw in Lemma \ref{lem:van-bou} the same is true for vanishing sequences with infinitesimal period. For these reasons in the next subsection we study the behaviour of the action form on the set of loops with short length.

\subsection{The subset of short loops}
We now define a local primitive for $\eta_k$ close to the subset of constant loops. For $k>e_0(L)$, such primitive will enjoy some properties that will enable us to apply the minimax theorem of Section \ref{sec:min} to prove the Main Theorem. For our arguments we will need estimates which hold uniformly on a compact interval of energies. Hence, for the rest of this subsection we will suppose that a compact interval $I\subset (e_0(L),+\infty)$ is fixed. 

Let $M_0\subset W^{1,2}_0$ be the constant loops parametrized by $\T$ and $M_0\times\R^+\subset \Lambda_0$ the constant loops with arbitrary period. We readily see that $\tau^\sigma|_{M_0}=0$. Thus, $\eta_k=dA_k|_{M_0\times\R^+}$ and
\begin{equation}\label{action-m0}
A_k(x,T)\ =\ T\,(k-V(x))\,,\quad\forall\,(x,T)\in M_0\times\R^+\,. 
\end{equation}
Now that we have described $\eta_k$ on constant loops, let us see what happens nearby. First, we need the following lemma.
\begin{lem}\label{def-ret}
There exists $\delta_*>0$ such that $\{\ell<\delta\}\subset W^{1,2}$ retracts with deformation on $M_0$, for all $\delta\leq\delta_*$. Thus, we have $\tau^\sigma|_{\{\ell<\delta_*\}}\, =\, dP^\sigma$, where
\begin{equation}
\begin{aligned}
P^\sigma:\{\ell<\delta_*\}&\ \longrightarrow\ \R\\
x&\ \longmapsto\ \int_{B^2}\hat u_x^*\sigma\,,
\end{aligned}
\end{equation}
where $\hat u_x:B^2\rightarrow M$ is the disc traced by $x$ under the action of the deformation retraction. Furthermore, there exists $C>0$ such that
\begin{equation}\label{psigma}
|P^\sigma(x)|\ \leq\ C\cdot \ell(x)^2\,.
\end{equation}
\end{lem}
\begin{proof}
Choose $\delta<2\rho(g)$, where $\rho(g)$ is the injectivity radius of $g$. With this choice, for each $x\in\{\ell<\delta\}$ and each $s\in\T$, there exists a unique geodesic $y_s:[0,1]\rightarrow M$ joining $x(0)$ to $x(s)$. For each $a\in[0,1]$ define $x_a:\T\rightarrow M$ by $x_a(s):=y_s(a)$. Taking a smaller $\delta$ if necessary, one can prove that $a\mapsto|x_a'|$ is a non-decreasing family of functions (use normal coordinates at $x(0)$). Thus, $a\mapsto \ell(x_a)$ is non-decreasing as well and 
\begin{align*}
[0,1]\times\{\ell<\delta\}&\longrightarrow \{\ell<\delta\}\\
(a,x)&\longmapsto x_a
\end{align*}
yields the desired deformation. In order to estimate $P^\sigma$ is enough to bound the area of the deformation disc $\hat u_x$:
\begin{equation*}
\op{area}(\hat u_x)\leq\int_0^1\di a\int_0^1 \left|\frac{dy_s}{da}(a)\right|\cdot|x_a'(s)|\,\di s\leq \int_0^1\di a\int_0^1d(x(0),x(s))|x'(s)|\,\di s\leq \frac{\ell(x)}{2}\ell(x)\,.
\end{equation*}
\end{proof}
In view of this lemma, for all $\delta\in(0,\delta_*]$, we define the set
\begin{equation}
\mathcal V^{\delta}\, :=\ \{\ell<\delta\}\times\R^+\,\subset\,\Lambda_0
\end{equation}
and the function $S_k:\mathcal V^{\delta_*}\,\longrightarrow\,\R$ given by
\begin{equation}
S_k\,:=\ A_k\ -\ P^\sigma\circ\op{pr}_{W^{1,2}}\,. 
\end{equation}
Such a function is a primitive of $\eta_k$ on $\mathcal V^{\delta_*}$. By \eqref{est-quad}, it admits the following upper bound.
\begin{prp}
There exists $C>0$ such that, for every $\gamma\in \mathcal V^{\delta_*}$, there holds
\begin{equation}\label{boundsk}
 S_k(\gamma)\ \leq\ C\cdot\left(\frac{e(x)}{T}\ +\ T\ +\ \ell(x)^2\right)\,,\quad\forall\, k\in I\,.
\end{equation}
\end{prp}
\noindent This result has an immediate consequence on vanishing sequences and flow lines of $\Phi^k$.
\begin{cor}\label{cor-perbel}
Let $b>0$ and $k\in I$ be fixed. The following two statements hold:
\begin{enumerate}
\item if $(\gamma_m)$ is a vanishing sequence such that $\gamma_m\notin\{S_k<b\}$ for all $m\in\N$, then $T_m$ is bounded away from zero;
\item the flow $\Phi^k$ is positively complete on the set $\Lambda\setminus\{S_k<b\}$.
\end{enumerate}
\end{cor}

We conclude this section by showing that the infimum of $S_k$ on short loops is zero and it is approximately achieved on constant loops with small period. Furthermore, $S_k$ is bounded away from zero on the set of loops having some fixed positive length. 
\begin{prp}\label{prp-mp}
There exist $\delta_I\leq \delta_*$ and positive numbers $b_I,T_I$ such that, for all $k\in I$,
\begin{equation}\label{mp-k}
\textit{(a)}\ \ \inf_{\mathcal V^{\delta_{I}}}S_k\ =\ 0\,,\quad\quad\textit{(b)}\ \ \inf_{\partial \mathcal V^{\delta_{I}}}S_k\ \geq\ b_I\,,\quad\quad\textit{(c)}\ \ \sup_{M_0\times\{T_{I}\}}S_k\ <\ \frac{b_I}{2}\,.
\end{equation}
\end{prp}
\begin{proof}
Since for all $q\in M$, the function $L|_{T_qM}$ attains its minimum at $(q,0)$, the estimate from below on $L$ obtained in \eqref{est-quad} can be refined to
\begin{equation*}
L(q,v)\ \geq\ C_1|v|^2\,+\,\min_{q\in M} L(q,0)\ =\ C_1|v|^2\,-\,e_0(L)\,.
\end{equation*}
From this inequality and \eqref{psigma}, we can bound from below $S_k(\gamma)$:
\begin{align*}
S_k(\gamma)\ &\geq\ T\cdot\int_0^1\big[C_1\cdot\frac{|x'|^2}{T^2}-e_0(L)+k\Big]\di s\,-\,C \cdot\ell(x)^2\\
&\geq\ C_1\cdot\frac{e(x)}{T}+(k-e_0(L))\cdot T\,-\,C\cdot \ell(x)^2\\
&\stackrel{\mbox{}^{(\star)}}{\geq}\ 2\sqrt{C_1(\min I-e_0(L))}\cdot\ell(x)\,-\,C\cdot \ell(x)^2\,.
\end{align*}
where in $(\star)$ we made use of the inequality between arithmetic and geometric mean. Hence, there exists $\delta_I>0$ sufficiently small, such that the last quantity is positive if $\ell(x)<\delta_I$ and bounded from below by
\begin{equation*}
b_I\,:=\ 2\sqrt{C_1(\min I-e_0(L))}\cdot\delta_I\,-\,C\cdot \delta_I^2\ >0
\end{equation*}
if $\ell(x)=\delta_I$. This implies Inequality \textit{(b)} in \eqref{mp-k} and that $\inf_{V^{\delta_{I}}}S_k\geq0$. To prove that $\inf_{V^{\delta_{I}}}S_k\leq0$ and that there exists $T_I$ such that Inequality \textit{(c)} in \eqref{mp-k} holds, we just recall from \eqref{action-m0} that
\begin{equation*}
\lim_{T\rightarrow0}\sup_{M_0\times\{T\}}S_k\ =\ 0\,.\qedhere
\end{equation*}
\end{proof}
In the next section we will prove a minimax theorem for a class of closed $1$-form on abstract Hilbert manifolds. Such a class will satisfy a general version of the properties we have proved so far for $\eta_k$.
\bigskip

\section{The minimax technique}\label{sec:min}
In this section we present an abstract minimax technique which represents the core of the proof of the Main Theorem. We formulate it in a very general form on a non-empty Hilbert manifold $\mathscr H$.

\subsection{An abstract theorem}\label{sub-minimax}

We start by setting some notation for homotopy classes of maps from Euclidean balls into $\mathscr H$. Let $d\in\N$ and $\mathscr U$ be a subset of $\mathscr H$. Define $\big[(B^d,\partial B^d),(\mathscr H,\mathscr U)\big]$ as the set of homotopy classes of maps $\gamma:(B^d,\partial B^d)\rightarrow(\mathscr H,\mathscr U)$. By this we mean that the maps send $B^d$ to $\mathscr H$ and $\partial B^d$ to $\mathscr U$, and that the homotopies do the same. The classes $[\gamma]$, where $\gamma$ is such that $\gamma(B^d)\subset \mathscr U$ are called \textit{trivial}. If $\mathscr U'\subset\mathscr U$, we have a map
\begin{equation*}
 i^{\mathscr U'}_{\mathscr U}:\big[(B^d,\partial B^d),(\mathscr H,\mathscr U')\big]\longrightarrow \big[(B^d,\partial B^d),(\mathscr H,\mathscr U)\big]
\end{equation*}
We are now ready to state the main result of this section.
\begin{thm}\label{thm-for}
Let $\mathscr H$ be a non-empty Hilbert manifold, $\mathscr I=[k_0,k_1]$ be a compact interval and $d\geq1$ an integer. Let $\alpha_k\in\Omega^1(\mathscr H)$ be a family of Lipschitz-continuous forms parametrized by $k\in\mathscr I$ and such that
\begin{itemize}
 \item the integral of $\alpha_k$ over contractible loops vanishes;
 \item $\alpha_k=\alpha_{k_0}+(k-k_0)d\mathscr T$, where $\mathscr T:\mathscr H\rightarrow(0,+\infty)$ is a $C^{1,1}$ function such that
 \begin{equation}
 \sup_{\mathscr H}|d\mathscr T|\ < \ +\infty\,.
 \end{equation}
\end{itemize}
Define the vector field
\begin{equation}
\mathscr X_k\,:=\ -\,\frac{\sharp\, \alpha_k}{\sqrt{1+|\alpha_k|^2}}\,,
\end{equation}
where $\sharp$ is the metric duality, and suppose that there exists an open set $\mathscr V\subset\mathscr H$ such that:
\begin{itemize}
 \item there exists $\mathscr S_k:\overline{\mathscr V}\rightarrow\R$ satisfying
 \begin{equation}\label{eq:prim}
d\mathscr S_k\ = \ \alpha_k\,,\quad\quad \mathscr S_k\ =\ \mathscr S_{k_0}\ +\ (k-k_0)\,\mathscr T\,;
 \end{equation}
 \item there exists a real number
 \begin{equation}\label{eq:infbeta}
\beta_0\ < \ \inf_{\partial \mathscr V}\mathscr S_{k_0}\ =:\,\beta_{\partial\mathscr V}
\end{equation}
such that the flow $r\mapsto\Phi^{\mathscr X_k}_r$ is positively complete on the set $\mathscr H\setminus\{\mathscr S_k<\beta_0\}$;
 \item there exists a set $\mathscr M\subset\{\mathscr S_{k_1}<\beta_0\}$ and a class $\mathscr G\in[(B^d,\partial B^d),(\mathscr H,\mathscr M)\big]$ such that $i^{\mathscr M}_{\mathscr V}(\mathscr G)$ is non-trivial.
\end{itemize}
Then, the following two statements hold true. First, for all $k\in\mathscr I$, there exists a sequence $(h^k_m)_{m\in\N}\subset\mathscr H\setminus\{\mathscr S_k< \beta_0\}$ such that 
\begin{equation*}
\lim_{m\rightarrow\infty}|\alpha_k|_{h^k_m}\ =\ 0\,.
\end{equation*}
Second, there exists a subset $\mathscr I_*\subset\mathscr I$ such that
\begin{itemize}
 \item $\mathscr I\setminus\mathscr I_*$ is negligible with respect to the $1$-dimensional Lebesgue measure;
 \item for all $k\in\mathscr I_*$ we have
 \begin{equation*}
  \sup_{m\in\N}\mathscr T(h^k_m)\ <\ +\infty\,.
 \end{equation*}
\end{itemize}
Moreover, if there exists a $C^{1,1}$-function $\widehat{\mathscr S}_k:\mathscr H\rightarrow\R$ which extends $\mathscr S_k$ and satisfies \eqref{eq:prim} on the whole $\mathscr H$, we also have that
\begin{equation}
\lim_{m\rightarrow\infty}\widehat{\mathscr S}_k(h^k_m)\ =\ \inf_{\gamma\in\mathscr G}\sup_{\xi\in B^d}\ \widehat{\mathscr S}_k\circ \gamma \,(\xi)\ \geq\ \beta_{\partial \mathscr V}\,.
\end{equation}
\end{thm}
To prove Theorem \ref{thm:main}\textit{(1a)} we will also need a version of the minimax theorem for $d=0$, namely when the maps are simply points in $\mathscr H$. We state it here for a single function and not for a $1$-parameter family since this will be enough for the intended application. For a proof we refer to \cite[Remark 1.10]{Abb13}.
\begin{thm}\label{thm:fun}
Let $\mathscr H$ be a non-empty Hilbert manifold and let $\widehat{\mathscr S}:\mathscr H\rightarrow \R$ be a $C^{1,1}$-function bounded from below. Suppose that the flow of the vector field
\begin{equation*}
\mathscr X\,:=\ -\,\frac{\nabla \widehat{\mathscr S}}{\sqrt{1+|\nabla\widehat{\mathscr S}|^2}}
\end{equation*}
is positively complete on some non-empty sublevel set of $\widehat{\mathscr S}$.
Then, there exists a sequence $(h_m)_{m\in\N}\subset\mathscr H$ such that
\begin{equation}
\lim_{m\rightarrow+\infty}|d_{h_m}\widehat{\mathscr S}|\ =\ 0\,,\quad\quad\lim_{m\rightarrow+\infty}\widehat{\mathscr S}(h_m)\ =\ \inf_{\mathscr H}\widehat{\mathscr S}\,.
\end{equation}
\end{thm}

\noindent In the next two subsections we prove Theorem \ref{thm-for}. First, we introduce some preliminary definitions and lemmas and then we present the core of the argument.

\subsection{Preliminary results}
We start by defining the \textit{variation} of the $1$-form $\alpha_k$ along any path $u:[a_0,a_1]\rightarrow \mathscr H$. It is the real number
\begin{align}\label{deltas}
\alpha_k(u)\,:=\ \int_{a_0}^{a_1}\alpha_k\left(\frac{du}{da}\right)(u(a))\,\di a\,.
\end{align}
We collect the properties of the variation along a path in a lemma.
\begin{lem}\label{lem:prim}
If $u$ is a path in $\mathscr H$ and $\overline{u}$ is the inverse path, we have
\begin{equation}
\alpha_k(\overline{u})\ =\ -\, \alpha_k(u)\,.
\end{equation}
If $u_1$ and $u_2$ are two paths in $\mathscr H$ such that the ending point of $u_1$ coincides with the starting point of $u_2$, we denote by $u_1\ast u_2$ the concatenation of the two paths and we have
\begin{equation}
\alpha_k(u_1\ast u_2)\ =\ \alpha_k(u_1)\ +\ \alpha_k(u_2)\,,
\end{equation}
If $u$ is a contractible closed path in $\mathscr H$, we have
\begin{equation}
\alpha_k(u)\ =\ 0\,.
\end{equation}
Finally, let $\gamma:Z\rightarrow \mathscr H$ be any smooth map from a Hilbert manifold $Z$ such that there exists a function $\mathscr S_k^\gamma: Z\rightarrow \mathscr H$ with the property that
\begin{equation}
d\,\mathscr S_k^\gamma\ =\ \gamma^*\alpha_k\,.
\end{equation}
Then, for all paths $z:[a_0,a_1]\rightarrow Z$ we have
\begin{equation}\label{eq:prim2}
\alpha_k(\gamma\circ z)\ =\ \mathscr S_k^\gamma(z(a_1))\ -\ \mathscr S_k^\gamma(z(a_0))\,.
\end{equation}
\end{lem}
Let us come back to the statement of Theorem \ref{thm-for}. Fix a point $\xi_*\in \partial B^d$ and for every $\gamma\in\mathscr G$ define the unique $\mathscr S_k^\gamma:B^d\rightarrow\mathscr H$ such that
\begin{equation}\label{eq:prim3}
d\,\mathscr S_k^\gamma\ =\ \gamma^*\alpha_k\,,\quad\quad \mathscr S_k^\gamma(\xi_*)\ =\ \mathscr S_k(\gamma(\xi_*))\,.
\end{equation}
We observe that this is a good definition since $B^d$ is simply connected and $\gamma(\xi_*)$ belongs to the domain of definition of $\mathscr S_k$ as $\gamma\in\mathscr G$. Moreover, if $\alpha_k$ admits a global primitive $\widehat {\mathscr S}_k$ on $\mathscr H$ extending $\mathscr S_k$, then clearly we have $\mathscr S_k^\gamma=\widehat{\mathscr S}_k\circ\gamma$. Finally, thanks to the previous lemma, for every $\xi\in B^d$ we have the formula
\begin{equation}\label{eq:primvar}
\mathscr S^\gamma_k(\xi)\ =\ \mathscr S_k(\gamma(\xi_*))\ +\ \alpha_k(\gamma\circ z_\xi)\,,
\end{equation}
where $z_\xi:[0,1]\rightarrow B^d$ is any path connecting $\xi_*$ and $\xi$.

\begin{rmk}
If $d\neq 1$, then $\mathscr S^\gamma_k$ does not depend on the choice of the point $\xi_*\in \partial B^d$ as $S^{d-1}=\partial B^d$ is connected. On the other hand, if $d=1$ there are two possible choices for $\xi_*$ and the two corresponding primitives of $\gamma^*\eta_k$ differ by a constant, which depends only on the class $\mathscr G$ and not on $\gamma$.
\end{rmk}

\begin{dfn}
We define the $\mathsf{minimax\ function}$ $c_{\mathscr G}:\mathscr I\rightarrow\R\cup\{-\infty\}$ by
\begin{align}
c_{\mathscr G}(k)\,:=\ \inf_{\gamma\in\mathscr G}\ \sup_{\xi\in B^d}\ \mathscr S_k^\gamma(\xi)\,.
\end{align}
\end{dfn}
\noindent In the next lemma we show that $c_{\mathscr G}(k)$ is finite and that, for each $\gamma\in\mathscr G$, the points almost realizing the supremum of the function $\mathscr S_k^\gamma$ lie in the complement of the set $\{\mathscr S_k\,<\,\beta_0\}$. 
\begin{lem}\label{lem:almax}
Let $k\in\mathscr I$ and $\gamma\in\mathscr G$. There holds
\begin{equation}\label{eq:bel}
\sup_{B^d}\ \mathscr S^\gamma_k\ \geq\ \beta_{\partial \mathscr V}\,.
\end{equation}
Moreover, if $\beta_1<\beta_{\partial\mathscr V}$, then $\forall\,\xi\in B^d$ the following implication holds
\begin{equation}\label{eq:almax}
\mathscr S_k^\gamma(\xi)\ \geq\ \sup_{B^d}\ \mathscr S_k^\gamma\ -\ (\beta_{\partial\mathscr V}-\beta_1)\quad\xRightarrow{\quad\ \ }\quad \gamma(\xi)\ \notin\ \{\mathscr S_k\,<\,\beta_1\}\,.
\end{equation}
\end{lem}
\begin{proof}
Since  $i^{\mathscr M}_{\mathscr V}(\mathscr G)$ is non-trivial, the set $\{\xi\in B^d\,|\, \gamma(\xi)\in\partial \mathscr V\}$ is non-empty. Therefore, there exists an element $\widehat\xi$ in this set and a path $z_{\widehat\xi}:[0,1]\rightarrow B^d$ from $\xi_*$ to $\widehat\xi$ such that $\gamma\circ z_{\widehat\xi}|_{[0,1)}\subset \mathscr V$. By \eqref{eq:primvar} and \eqref{eq:prim2} we have
\begin{equation*}
\mathscr S^\gamma_k(\widehat\xi)\ =\ \mathscr S_k(\gamma(\xi_*))\,+\,\alpha_k(\gamma\circ z_{\widehat\xi})\ =\ \mathscr S_k(\gamma(\xi_*))\,+\,\Big(\mathscr S_k(\gamma(\widehat\xi))-\mathscr S_k(\gamma(\xi_*))\Big)\ =\ \mathscr S_k(\gamma(\widehat\xi))\,,
\end{equation*}
which implies \eqref{eq:bel} by \eqref{eq:infbeta}. In order to prove the second statement we consider $\xi\in B^d$ such that $\gamma(\xi)\in\{\mathscr S_k\,<\beta_1\}$. Without loss of generality there exists a path $z_{\xi,\widehat\xi}:[0,1]\rightarrow B^d$ from $\xi$ to $\widehat\xi$ such that $z_{\xi,\widehat\xi}|_{[0,1)}\subset\mathscr V$. Using \eqref{eq:prim2} twice, we compute
\begin{align*}
\sup_{B_d}\mathscr S^\gamma_k\ \geq\ \mathscr S^\gamma_k(\widehat\xi)\ =\ \mathscr S_k^\gamma(\xi)\,+\,\alpha_k(\gamma\circ z_{\xi,\widehat\xi})\ &=\ \mathscr S_k^\gamma(\xi)\,+\,\Big(\mathscr S_k(\gamma(\widehat\xi))-\mathscr S_k(\gamma(\xi))\Big)\\
&>\ \mathscr S_k^\gamma(\xi)\,+\,\big(\beta_{\partial \mathscr V}-\beta_1\big)\,,
\end{align*}
which yields the contrapositive of the implication we had to show.
\end{proof}
We now see that, since the family $k\mapsto\alpha_k$ is monotone in the parameter $k$, the same is true for the numbers $c_{\mathscr G}(k)$.
\begin{lem}
If $k_2\leq k_3$ and $\gamma\in\mathscr G$, we have
\begin{equation}\label{fin-dif}
\mathscr S^\gamma_{k_3}\ =\ \mathscr S^\gamma_{k_2}\ +\ (k_3-k_2)\,\mathscr T\circ \gamma\,.
\end{equation}
As a consequence, $c_{\mathscr G}$ is a non-decreasing function.
\end{lem}
\begin{proof}
We observe that
\begin{align*}
\bullet&\ \ d\big(\mathscr S^\gamma_{k_3}\,-\,\mathscr S^\gamma_{k_2}\big)\ =\ \gamma^*\big(\alpha_{k_3}-\alpha_{k_2}\big)\ =\ \gamma^*\big((k_3-k_2)\, d\mathscr T\big)\\
\bullet&\ \ \mathscr S^\gamma_{k_3}(\xi_*)\,-\, \mathscr S^\gamma_{k_2}(\xi_*)\ =\ \mathscr S_{k_3}(\gamma(\xi_*))\,-\mathscr S_{k_2}(\gamma(\xi_*))\ =\ (k_3-k_2)\,\mathscr T(\gamma(\xi_*))\,.
\end{align*}
These two equalities imply that the function $\mathscr S_{k_2}^\gamma+(k_3-k_2)\mathscr T\circ\gamma$ satisfies \eqref{eq:prim3} with $k=k_3$. Since these conditions identify a unique function, equation \eqref{fin-dif} follows. In particular, we have $\mathscr S^\gamma_{k_2}\leq\mathscr S^\gamma_{k_3}$. Taking the inf-sup of this inequality on $\mathscr G$, we get $c_{\mathscr G}(k_2)\leq c_{\mathscr G}(k_3)$.
\end{proof}

We end this subsection by adjusting the vector field $\mathscr X_{k}$ so that its flow becomes positively complete on all $\mathscr H$. We fix $\beta_1\in(\beta_0,\beta_{\partial \mathscr V})$ and let $\mathscr B:[\beta_0,\beta_1]\rightarrow[0,1]$ be a function that is equal to $0$ in a neighbourhood of $\beta_0$ and equal to $1$ in a neighbourhood of $\beta_1$. We set
\begin{equation*}
\check {\mathscr X}_{k}\,:=\ (\mathscr B\circ\mathscr S_{k})\cdot \mathscr X_{k}\,\in\,\Gamma(\mathscr H)\,.
\end{equation*}
We observe that 
\begin{equation*}
\bullet\ \check {\mathscr X}_{k}\,=\,0\ \ \mbox{on}\ \ \{\mathscr S_{k}<\beta_0\}\,,\quad\quad\bullet\ \check {\mathscr X}_{k}\,=\,\mathscr X_{k}\ \ \mbox{on}\ \ \mathscr H\setminus\{\mathscr S_{k}<\beta_1\}\,,
\end{equation*}
and, hence, the flow $\Phi^{\check {\mathscr X}_{k}}$ is positively complete.
\bigskip

\subsection{Proof of Theorem \ref{thm-for}}

Let us define the subset 
\begin{equation*}
\mathscr I_*:=\, \Big\{\,k\in[k_0,k_1)\ \Big|\ \exists\, C(k_*)\, \mbox{ such that }\, c_{\mathscr G}(k)-c_{\mathscr G}(k_*)\, \leq\,  C(k_*)(k-k_*)\,,\ \forall\, k\in[k_*,k_1]\,\Big\}\,.
\end{equation*}
Namely, $\mathscr I_*$ is the set of points at which $c_{\mathscr G}$ is Lipschitz-continuous on the right. Since $c_{\mathscr G}$ is a non-decreasing real function, by Lebesgue Differentiation Theorem, $c_{\mathscr G}$ is Lipschitz-continuous at almost every point. In particular, $\mathscr I\setminus\mathscr I_*$ has measure zero. 

We are now ready to show that
\begin{enumerate}
 \item for all $k\in\mathscr I$, there exists a vanishing sequence $(h^k_m)_{m\in\N}\subset\mathscr H\setminus\{\mathscr S_k<\beta_0\}$ and that
 \item for all $k_*\in\mathscr I_*$, such vanishing sequence can be taken to satisfy 
\begin{equation*}
\sup_{m\in\N}\mathscr T(h^{k_*}_m)\ <\ C(k_*)\ +\ 3\,.
\end{equation*}
\end{enumerate}
We will prove only the statement about the vanishing sequences with parameter in $\mathscr I_*$, as the argument can be easily adapted to prove the statement for a general parameter in $\mathscr I$.

We assume by contradiction that there exists a positive number $\varepsilon_0$ such that
\begin{equation}
|\alpha_{k_*}|\ \geq\ \varepsilon_0\,,\quad \mbox{on}\ \ \{\mathscr T\,<\,C(k_*)+3\}\setminus\{\mathscr S_{k_*}\,<\,\beta_1\}\,. 
\end{equation}
Consider a decreasing sequence $(k_m)_{m\in\N}\subset(k_*,k_1]$ such that $k_m\rightarrow k_*$. Set $\delta_m:=k_m-k_*$ and take a corresponding sequence $(\gamma_m)_{m\in\N}\subset\mathscr G$ such that
\begin{equation*}
\sup_{B^d}\mathscr S^{\gamma_m}_{k_m}\ <\ c_{\mathscr G}(k_m)\ +\ \delta_m\,.
\end{equation*}
For every $\xi\in B^d$ we consider the sequence of flow lines 
\begin{align*}
u_m^\xi:[0,1]&\longrightarrow\mathscr H\\
r&\longmapsto \Phi^{\check {\mathscr X}_{k_*}}_r(\gamma_m(\xi))\,.
\end{align*}
Conversely, for any time parameter $r\in[0,1]$, we get the map
\begin{equation}
\gamma^r_m\,:=\ \Phi^{\check {\mathscr X}_{k_*}}_r(\gamma_m)\,.
\end{equation}
We readily see that $\gamma^r_m|_{\partial B^d}=\gamma_m|_{\partial B^d}$ and $\gamma^r_m\in\mathscr G$. In particular, for every $\xi\in B^d$ and $r\in [0,1]$ the concatenated curve
\begin{equation}
\big(\gamma_m\circ z_{\xi}\big)\ \ast\ u^\xi_m|_{[0,r]}\ \ast\ \big(\overline{\gamma^r_m\circ z_\xi}\big)
\end{equation}
is contractible. Therefore, Lemma \ref{lem:prim} and Equation \eqref{eq:primvar} yield
\begin{equation}\label{deltahomo}
\mathscr S_{k_*}^{\gamma^r_m}(\xi)\ =\ \mathscr S_{k_*}^{\gamma_m}(\xi)\ +\ \alpha_{k_*}(u_m^\xi|_{[0,r]})\,.
\end{equation}
Finally, since $u_m^\xi$ is a flow line, we have
\begin{equation}\label{deltav}
\alpha_{k_*}(u_m^\xi|_{[0,r]})=\int_0^r\alpha_{k_*}\left(-\frac{\mathscr B\cdot\sharp\alpha_{k_*}}{\sqrt{1+|\alpha_{k_*}|^2}}\right)\!(u_m^\xi(\rho))\,\di \rho=-\int_0^r\frac{\mathscr B\cdot|\alpha_{k_*}|^2}{\sqrt{1+|\alpha_{k_*}|^2}}(u_m^\xi(\rho))\,\di \rho\,.
\end{equation}
Therefore $\alpha_{k_*}(u_m^\xi|_{[0,r]})\leq0$ and we find that, for every $m\in\N$,
\begin{equation}\label{eq:noninc}
r\longmapsto \mathscr S^{\gamma^r_m}_{k_*}\quad \mbox{is a non-increasing family of functions on } B^d\,.
\end{equation}

Let us estimate the supremum of $\mathscr S_{k_*}^{\gamma^r_m}$. When $r=0$, \eqref{fin-dif} and the definition of $\mathscr I_*$ imply:
\begin{equation}\label{supa0}
\sup_{B^d}\mathscr S_{k_*}^{\gamma_m}\ \leq\ \sup_{B^d}\mathscr S_{k_m}^{\gamma_m}\ <\ c_{\mathscr G}(k_m)+\delta_m\ \leq\ c_{\mathscr G}(k_*)+(C(k_*)+1)\,\delta_m\,.
\end{equation}
Thus, by \eqref{eq:noninc} we get, for every $r\in[0,1]$,
\begin{equation}\label{supuma}
\sup_{B^d}\mathscr S_{k_*}^{\gamma_m^r}\ <\ c_{\mathscr G}(k_*)+(C(k_*)+1)\,\delta_m\,.
\end{equation}

If $r\in[0,1]$, we define the sequence of subsets of $B^d$
\begin{align*}
J^r_m\,:&=\ \big\{\,\mathscr S_{k_*}^{\gamma^r_m}\ >\ c_{\mathscr G}(k_*)\,-\,\delta_m\,\big\}\,.
\end{align*}
Let us give a closer look to these sets. First, we observe that if $\xi\in J^r_m$, then \eqref{deltahomo} and \eqref{supuma} imply that
\begin{equation}\label{eq:vardec}
\alpha_{k_*}(u^\xi_m|_{[0,r]})\ >\ c_{\mathscr G}(k_*)-\delta_m\ -\ \big(\,c_{\mathscr G}(k_*)+(C(k_*)+1)\,\delta_m\,\big)\ =\ -\,(C(k_*)+2)\,\delta_m\,.
\end{equation}
Then, we claim that for $m$ large enough
\begin{equation}
\xi\in J^r_m\quad\Longrightarrow\quad \gamma^r_m(\xi)\ \in\ \big\{\mathscr T<C(k_*)+3\big\}\setminus\big\{\mathscr S_{k_*}<\beta_1\big\}\,,\quad\forall\,r\in[0,1]\,.
\end{equation}
First, we observe that
\begin{equation}
\mathscr S_{k_*}^{\gamma_m^r}(\xi)\ >\ c_{\mathscr G}(k_*)\ -\ \delta_m\ \geq\ \sup_{B^d}\mathscr S_{k_*}^{\gamma_m^r}\ -\ (C(k_*)+2)\,\delta_m\,.
\end{equation}
If $m$ is large enough, then $(C(k_*)+2)\,\delta_m<(\beta_{\partial \mathscr V}-\beta_1)$ and Lemma \ref{lem:almax} implies that $\gamma^r_m(\xi)\notin\{\mathscr S_{k_*}<\beta_1\}$. As a by-product we get that $u^\xi_m|_{[0,r]}$ is a genuine flow line of $\Phi^{{\mathscr X}_{k_*}}$. Then, we estimate $\mathscr T(\gamma^r_m(\xi))$. We start by taking $r=0$. In this case from \eqref{fin-dif} we get
\begin{equation*}
\mathscr T(\gamma_m(\xi))\ =\ \frac{\mathscr S_{k_m}^{\gamma_m}(\xi)-\mathscr S_{k_*}^{\gamma_m}(\xi)}{\delta_m}\ <\ \frac{c_{\mathscr G}(k_m)+\delta_m-c_{\mathscr G}(k_*)+\delta_m}{\delta_m}\ <\  C(k_*)+2\,.
\end{equation*}
To prove the inequality for arbitrary $r$ we bound the variation of $\mathscr T$ along $u^\xi_m|_{[0,r]}$ in terms of the action variation:
\begin{align*}
 -\alpha_{k_*}(u_m^\xi|_{[0,r]})\ =\ -\int_0^{r}\alpha_{k_*}\left(\frac{du^\xi_m}{d\rho}\right)\di \rho\ &\geq\ \int_0^{r}\left|\frac{du^\xi_m}{d\rho}\right|^2\di \rho\\
 &\geq\ \frac{1}{r}\left(\int_0^{r}\left|\frac{du^\xi_m}{d\rho}\right|\di \rho\right)^2\\
 &\geq\ \frac{1}{r}\left(\int_0^{r}\frac{1}{1+\sup_\mathscr H|d\mathscr T|}\left|\frac{d(\mathscr T\circ u^\xi_m)}{d\rho}\right|\di \rho\right)^2\\
 &\geq\ \frac{1}{r(1+\sup_\mathscr H|d\mathscr T|)^2}|\mathscr T(u^\xi_m(r))-\mathscr T(u^\xi_m(0))|^2\,.
\end{align*}
Using \eqref{eq:vardec} and rearranging the terms we get for $m$ large enough
\begin{equation*}
|\mathscr T(\gamma^r_m(\xi))-\mathscr T(\gamma_m(\xi))|^2\ \leq\ r\cdot(1+\sup_\mathscr H|d\mathscr T|)^2\cdot (C(k_*)+2)\,\delta_m\ <\ 1\,.
\end{equation*}
Hence, if $m$ is large enough the bound on $\mathscr T$ we were looking for follows from
\begin{equation}
\mathscr T(\gamma^r_m(\xi))\ \leq\ \mathscr T(\gamma_m(\xi))\ +\ |\mathscr T(\gamma^r_m(\xi))-\mathscr T(\gamma_m(\xi))|\ <\ (C(k_*)+2)\ +\ 1\,.
\end{equation}
The claim is thus completely established.

The last step to finish the proof of Theorem \ref{thm-for} is to show that $J^1_m=\emptyset$ for $m$ large enough. By contradiction, let $\xi\in J^1_m$. Since $\xi\in J^r_m$ for all $r\in[0,1]$, we see that $u_m^\xi$ is a flow line of $\Phi^{\mathscr X_{k_*}}$ contained in $\{\mathscr T<C(k_*)+3\}\setminus\{\mathscr S_{k_*}<\beta_1\}$. Using \eqref{eq:vardec} and continuing the chain of inequalities in \eqref{deltav}, we find
\begin{equation*}
-\,(C(k_*)+2)\,\delta_m\ <\ \alpha_{k_*}(u_m^\xi)\ \leq\ -\,\frac{\varepsilon_0^2}{\sqrt{1+\varepsilon_0^2}}
\end{equation*}
(where we used that the real function $w\mapsto \frac{w}{\sqrt{1+w}}$ is increasing). Such inequality cannot be satisfied for $m$ large, proving that the sets $J^1_m$ become eventually empty.

Finally, since $J^1_m=\emptyset$, we obtain that $c_{\mathscr G}(k_*)\leq\sup_{B^d}\mathscr S_{k_*}^{\gamma^1_m}\leq c_{\mathscr G}(k_*)-\delta_m$. This contradiction finishes the proof of Theorem \ref{thm-for}.
\medskip

In the next section we will determine when $\eta_k$ satisfies the hypotheses of the abstract theorem we have just proved.
\bigskip

\section{Proof of the Main Theorem}\label{sec:geo}
We now move to the proof of points \textit{(1)}, \textit{(2)}, \textit{(3)} of Theorem \ref{thm:main}. In the first preparatory subsection, we will see when the action form is exact.
\subsection{Primitives for $\eta_k$}
We know that $\eta_k$ is exact if and only if so is $\tau^\sigma$. The next proposition, whose simple proof we omit, gives necessary and sufficient conditions for the transgression form to be exact.
\begin{prp}
If $[\widetilde\sigma]\neq0$, then $\tau^\sigma|_{W^{1,2}_\nu}$ is not exact for any $\nu$.

If $[\widetilde\sigma]=0$, then
\begin{align*}
\widehat{P}^\sigma:W^{1,2}_0&\longrightarrow\R\,.\\
x&\longmapsto \int_{B^2}\hat u_x^*\sigma
\end{align*}
is a primitive for $\tau^\sigma$. Here $\hat u_x$ is any capping disc for $x$. This definition extends the primitive $P^\sigma$, which we constructed on the subset of short loops.

If $[\widetilde\sigma]_b=0$, then, given $\nu$ and a reference loop $x_\nu\in W^{1,2}_\nu$,
\begin{align*}
\widehat{P}^\sigma:W^{1,2}_\nu&\longrightarrow\R\,.\\
x&\longmapsto \int_{B^2}\hat u_{x_\nu,x}^*\sigma
\end{align*}
is a primitive for $\tau^\sigma$. Here $\hat u_{x_\nu,x}$ is a connecting cylinder from $x_\nu$ to $x$. If we take $x_0$ as a constant loop, the two definitions of $\widehat{P}^\sigma$ coincide on $W^{1,2}_0$.
\end{prp}
\begin{exe}
Show that if $M=\T^2$ and $[\sigma]\neq0$, then $\tau^\sigma|_{W^{1,2}_\nu}$ is not exact if $\nu\neq0$.
\end{exe}
We set $\widehat{S}_k:=A_k-\widehat{P}^\sigma\circ\op{pr}_{W^{1,2}}$ in the two cases above where $\widehat{P}^\sigma$ is defined. Theorem A in \cite{CIPP98} tells us when $\widehat{S}_k$ is bounded from below.
\begin{prp}\label{prp-below}
If $[\widetilde\sigma]=0$, then $\widehat{S}_k:\Lambda_0\rightarrow\R$ is bounded from below if and only if $k\geq c(L,\sigma)$. If $[\widetilde\sigma]_b=0$, the same is true for $\widehat{S}_k:\Lambda_\nu\rightarrow\R$.
\end{prp}
\begin{rmk}
Originally the critical value was introduced by Ma\~n\'e as the infimum of the values of $k$ such that $\widehat{S}_k:\Lambda_0\rightarrow\R$ is bounded from below \cite{Man97,CDI97}. Thus, the proposition above establishes the equivalence between the more geometric definition in \eqref{mandef} and the original one.
\end{rmk}
\begin{exe}
Prove that $\widehat S_k|_{\Lambda_\nu}$ is bounded from below if and only if $\widehat S_k|_{\Lambda_0}$ is bounded from below if and only if $\widehat S_k|_{\Lambda_0}$ is non-negative.
\end{exe}
As a by-product of Proposition \ref{prp-below}, we can give a criterion guaranteeing that a vanishing sequence for $\eta_k$ has bounded periods, provided $k>c(L,\sigma)$.
\begin{cor}\label{per-bou}
Let $\nu\in[\T,M]$ and $[\widetilde\sigma]_b=0$. If $k>c(L,\sigma)$ and $b\in\R$, then there exists a constant $C(\nu,k,b)$ such that
\begin{equation*}
\forall\, \gamma\in\Lambda_\nu\,,\quad \widehat{S}_k(\gamma)\ <\ b\ \ \Longrightarrow\ \ T\ <\ C(\nu,k,b)\,.
\end{equation*}
\end{cor}
\begin{proof}
We readily compute
\begin{equation*}
T\ =\ \frac{\widehat{S}_k(\gamma)-\widehat{S}_{c(L,\sigma)}(\gamma)}{k-c(L,\sigma)}\ \leq\ \frac{b-\inf_{\Lambda_\nu}\widehat{S}_{c(L,\sigma)}}{k-c(L,\sigma)}\ =:\ C(\nu,k,b)\,.\qedhere
\end{equation*}
\end{proof}

\subsection{Non-contractible orbits}
We now prove the existence of non-contractible orbits as prescribed by the Main Theorem.

\begin{proof}[Proof of Theorem \ref{thm:main}.$(1a)$]
Let $\nu\in[\T,M]$ be a non-trivial class, $\sigma$ be a magnetic form such that $[\sigma]_b=0$ and $k>c(L,\sigma)$. Thanks to Proposition \ref{prp-below}, the infimum of $\widehat{S}_k$ on $\Lambda_\nu$ is finite. Then, we apply Theorem \ref{thm:fun} with $\mathscr H=\Lambda_\nu$ and $\widehat{\mathscr S}=\widehat{S}_k$ and we obtain a vanishing sequence $(\gamma_m)_{m\in\N}$ such that $\widehat{S}_k(\gamma_m)$ is uniformly bounded. By Corollary \ref{per-bou} the sequence of periods is bounded from above. By Corollary \ref{cor-perbel} the sequence of periods is also bounded away from zero. Therefore, we can apply Proposition \ref{prp-conv} to get a limit point of the sequence.
\end{proof}
\bigskip

\subsection{Contractible orbits}

We start by recalling a topological lemma. 
\begin{prp}\label{prp-top}
If $d\geq1$ and $\delta\leq\delta_*$ (see Lemma \ref{def-ret}), there are natural bijections
\begin{equation*}
\xymatrixcolsep{35pt}\xymatrix{
\displaystyle\frac{\pi_{d+1}(M)}{\pi_1(M)}\ \ar[rd]^{F}\ar[r]&\ [\,S^{d+1},\,M\,]\ar[d]&\\
&\ \big[\,\big(B^d,\partial B^d\big)\,,\,\big(W^{1,2}_0, M_0\big)\,\big]\ \ar[r]^<<<<<{\quad i^{M_0}_{\{\ell<\delta\}}}&\ \big[\,\big(B^d,\partial B^d\big)\,,\,\big(W^{1,2}_0,\{\ell<\delta\}\big)\,\big]\,,}
\end{equation*}
where $\pi_{d+1}(M)/\pi_1(M)$ is the quotient of $\pi_{d+1}(M)$ by the action of $\pi_1(M)$\footnote{Here a choice of an arbitrary base point $q_0\in M$ is to be understood: $\pi_{d+1}(M):=\pi_{d+1}(M,q_0)$ and $\pi_1(M):=\pi_1(M,q_0)$}. The trivial classes on the second line are identified with the class of constant maps in $[S^{d+1},M]$ and with the class of the zero element in $\pi_{d+1}(M)/\pi_1(M)$.
\end{prp}
\begin{proof}
The first horizontal map is $\frac{[\hat u]}{\pi_1(M)}\mapsto [\hat u]$. We leave as an exercise to the reader to show that is a bijection. The vertical map sends $[\hat u]$ to $[u]$, where $u$ is defined as follows. Consider the equivalence relation $\sim$ on $B^d\times \T$:
\begin{equation}
(z_1,s_1)\,\sim\,(z_2,s_2)\quad\quad\Longleftrightarrow\quad\quad (z_1,s_1)\,=\,(z_2,s_2)\quad \vee\quad z_1\,=\,z_2\, \in\, \partial B^d\,.
\end{equation}
If we interpret $B^d$ as the unit ball in $\R^d$ and $S^{d+1}$ as the unit sphere in $\R^{d+2}$ we can define the homeomorphism
\begin{align*}
Q:\frac{B^d\times \T}{\sim}&\ \longrightarrow\ S^{d+1}\\
[z,s]&\ \longmapsto\ (z,\sqrt{1-|z|^2}\cdot e^{2\pi is})\,,
\end{align*}
where $e^{2\pi is}$ belongs to $S^1\subset\R^2$. We set $u(z)(s):=(\hat u\circ Q)([z,s])$. For a proof that the vertical map is well-defined and it is a bijection, we refer the reader to \cite[Proposition 2.1.7]{Kli78}. Finally, the second horizontal map is a bijection thanks to Lemma \ref{def-ret}.
\end{proof}
\bigskip

We can now prove the parts of the Main Theorem dealing with contractible orbits.
\begin{proof}[Proof of Theorem \ref{thm:main}.$(1b)$]
Let $[\widetilde\sigma]_b=0$, $k>c(L,\sigma)$ and fix some non-zero $\mathfrak u\in\pi_{d+1}(M)$, which exists by hypothesis. We apply Proposition \ref{prp-mp} to the trivial interval ${\{k\}}$ and get the positive real numbers $\delta_{\{k\}}$, $b_{\{k\}}$ and $T_{\{k\}}$. Let
\begin{equation}
\Gamma_{\mathfrak u}\,:=\ \Big\{\ \gamma=(x,T):\big(B^d,\partial B^d\big)\ \longrightarrow\ \big(\Lambda_0,M_0\times\{T_{\{k\}}\}\big)\ \ \Big|\ \ [x]\in F\big(\mathfrak u/\pi_1(M)\big)\ \Big\}
\end{equation}
By Proposition \ref{prp-top} we see that $\Gamma_{\mathfrak u}\in\big[(B^d,\partial B^d),(\Lambda_0,M_0\times\{T_{\{k\}}\})\big]$ and that $i^{M_0\times\{T_{\{k\}}\}}_{\mathscr V^{\delta_{\{k\}}}}(\Gamma_{\mathfrak u})$ is non-trivial. Therefore, we apply Theorem \ref{thm-for} with
\begin{equation*}
\left[\ \begin{aligned}
\mathscr H&=\Lambda_0&\quad \mathscr I&= \{k\}&\quad\widehat{\mathscr S}_k&=\widehat S_k\\
\beta_0&=b_{\{k\}}/2 &\quad \mathscr V&=\mathcal V^{\delta_{\{k\}}}&\quad\mathscr M&=M_0\times\{T_{\{k\}}\}
\\
\mathscr G&=\Gamma_{\mathfrak u} & & & &
\end{aligned}\ \right]
\end{equation*}
and we obtain a vanishing sequence $(\gamma_m)_{m\in \N}$ such that
\begin{equation*}
\lim_{m\rightarrow+\infty}\widehat{S}_k(\gamma_m)\ =\ c_{\mathfrak u}(k)\,:=\ \inf_{\gamma\in\Gamma}\ \sup_{B^d}\widehat{S}_k\circ \gamma\ \geq\ b_{\{k\}} \,.
\end{equation*}
The sequence of periods $(T_m)$ is bounded from above by Corollary \ref{per-bou}. The sequence $(T_m)$ is also bounded away from zero by Corollary \ref{cor-perbel}, since $\gamma_m\notin \{S_k< b_{\{k\}}/2\}$ for $m$ big enough. Applying Proposition \ref{prp-conv} we obtain a limit point of $(\gamma_m)$.
\end{proof}
\bigskip

\begin{proof}[Proof of Theorem \ref{thm:main}.$(2)$]
Let $[\widetilde\sigma]=0$ and fix $I=[k_0,k_1]\subset (e_0(L),c(L,\sigma))$. Let $\delta_I$, $b_I$ and $T_I$ be as in Proposition \ref{prp-mp}. Fix $\gamma_0\in M_0\times\{T_I\}$ and $\gamma_1\in\Lambda_0$ such that $\widehat{S}_{k_1}(\gamma_1)<0$. Such element exists thanks to Proposition \ref{prp-below}. Let $u_*:[0,1]\rightarrow \Lambda$ be some path such that $u_*(0)=\gamma_0$ and $u_*(1)=\gamma(1)$ and denote by $[u_*]\in[(B^1,\partial B^1),(\Lambda_0,\{\gamma_0,\gamma_1\})]$ its homotopy class. By Proposition \ref{prp-mp}, $\gamma_0$ and $\gamma_1$ belong to different components of $\{\widehat{S}_{k_0}<b_I\}$. Thus, $i^{\{\gamma_0,\gamma_1\}}_{\{\widehat{S}_{k_0}<b_I\}}([u_*])$ is non-trivial. Therefore, we apply Theorem \ref{thm-for} with
\begin{equation*}
\left[\ \begin{aligned}
\mathscr H&=\Lambda_0&\quad \mathscr I&= I&\quad\widehat{\mathscr S}_k&=\widehat S_k\\
\beta_0&=b_I/2 &\quad \mathscr V&=\{\widehat{S}_{k_0}<b_I\}&\quad\mathscr M&=\{\gamma_0,\gamma_1\}
\\
\mathscr G&=[u_*] & & & &
\end{aligned}\ \right]
\end{equation*}
and we get a vanishing sequence $(\gamma^k_m)_{m\in\N}$ with bounded periods, for almost every $k\in I$. Moreover, we have
\begin{equation*}
\lim_{m\rightarrow+\infty}\widehat{S}_k(\gamma_m)\ =\ c_{[u_*]}(k)\,:=\ \inf_{u\in[u_*]}\ \sup_{B^1}\widehat{S}_k\circ u\ \geq\ b_I\,. 
\end{equation*}
In particular, $\gamma^k_m\notin\{\widehat{S}_k <b_{I}/2\}$ for $m$ large enough. Hence, the periods are bounded away from zero by Corollary \ref{cor-perbel}. Now we apply Proposition \ref{prp-conv} to get a limit point of $(\gamma^k_m)$. Taking an exhaustion of $(e_0(L),c(L,\sigma))$ by compact intervals, we get a critical point for almost every energy in $(e_0(L),c(L,\sigma))$.
\end{proof}
\bigskip

\begin{proof}[Proof of Theorem \ref{thm:main}.$(3)$]
Let $[\widetilde\sigma]\neq0$ and fix $I=[k_0,k_1]\subset (e_0(L),+\infty)$. Let $\delta_I$, $b_I$ and $T_I$ be as in Proposition \ref{prp-mp}. Since $[\widetilde\sigma]\neq0$, there exists a non-zero $\mathfrak u\in\pi_2(M)$. We set
\begin{equation}
\Gamma_{\mathfrak u}\,:=\ \Big\{\ \gamma=(x,T):\big(B^1,\partial B^1\big)\ \longrightarrow\ \big(\Lambda_0,M_0\times\{T_I\}\big)\ \ \Big|\ \ [x]\in F\big(\mathfrak u/\pi_1(M)\big)\ \Big\}
\end{equation}
By Proposition \ref{prp-top} we see that $\Gamma_{\mathfrak u}\in\big[(B^1,\partial B^1),(\Lambda_0,M_0\times\{T_I\})\big]$ and that $i^{M_0\times\{T_I\}}_{\mathcal V^{\delta_I}}(\Gamma_{\mathfrak u})$ is non-trivial. Therefore, we apply Theorem \ref{thm-for} with
\begin{equation*}
\left[\ \begin{aligned}
\mathscr H&=\Lambda_0&\quad \mathscr I&= I&\quad\alpha_k&=\eta_k\\
\beta_0&=b_I/2 &\quad \mathscr V&=\mathcal V^{\delta_I}&\quad\mathscr M&=M_0\times\{T_I\}
\\
\mathscr G&=\Gamma_{\mathfrak u} & & & &
\end{aligned}\ \right]
\end{equation*}
and we obtain a vanishing sequence $(\gamma^k_m)_{m\in\N}\subset \Lambda_0\setminus \{S_k<b_I/2\}$ with bounded periods, for almost every $k\in I$. Since, the periods are bounded away from zero by Corollary \ref{cor-perbel}, Proposition \ref{prp-conv} yields a limit point of $(\gamma^k_m)$, for almost every $k\in I$. 

Taking an exhaustion of $(e_0(L),+\infty)$ by compact intervals, we get a contractible zero of $\eta_k$ for almost every $k>e_0(L)$.
\end{proof}

\section{Magnetic flows on surfaces I: Ta\u\i manov minimizers}\label{sec:tai}
In this and in the next section we are going to focus on the $2$-dimensional case. Therefore, let us assume that $M$ is a closed connected oriented surface. In this case $H^2(M;\R)\simeq \R$, where the isomorphism is given by integration and we identify $[\sigma]$ with a real number. Up to changing the orientation on $M$, we assume that $[\sigma]\geq0$. 

For simplicity, we are going to work in the setting of Section \ref{sub:hom} and consider only purely kinetic Lagrangians. Namely, we take $L(q,v)=\frac{1}{2}|v|^2$, where $|\cdot|$ is induced by a metric $g$.

Since $L$ depends only on $g$, we will use the notation $(g,\sigma)$ where we previously used $(L,\sigma)$. We readily see that $e_m(L)=e_0(L)=0$ and that $c(g,\sigma)=0$ if and only if $\sigma=0$ (see Proposition \ref{prp-man}).
We recall that the periodic orbits with positive energy are parametrized by a positive multiple of the arc-length. Thus, they are immersed curve in $M$. 

\subsection{The space of embedded curves}
The space of curves on a $2$-dimensional manifold $M$ has a particularly rich geometric structure. Observe, indeed, that for $n\geq3$ the curves on $M$ are generically embedded. On the other hand, if $M$ is a surface, intersections between curves and self-intersections are generically stable. Therefore, one can refine the existence problem by looking at periodic orbits having a particular shape (see the beginning of Section 1.1 in \cite{HS13} and references therein for a precise notion of the shape of a curve on a surface). For example, we consider the following question.
\begin{center}
 For which $k$ and $\nu$ there exists a \textit{simple} periodic orbit $\gamma\in\Lambda_\nu$ with energy $k>0$? 
\end{center}
Let us start by investigating the case $\nu=0$. If $\gamma=(x,T)$ is a contractible simple curve, there exists an embedded disc $\hat u:B^2\rightarrow M$ such that $\hat u(e^{2\pi is})=x(s)$. This map yields a path $(u,T)$ in $\Lambda_0$ from a constant path $(x_0,T)$, representing the centre of the disc, to $(x,T)$. Integrating $\eta_k$ along this path and summing the value of $S_k$ at $(x_0,T)$, we get
\begin{equation}\label{int-emb}
 \int_0^1(u,T)^*\eta_k\ +\ S_k(x_0,T)\ =\ \frac{e(x)}{2T}\ +\ kT\ -\ \int_{B^2}\hat u^*\sigma\,.
\end{equation}
Since $\hat u$ is an embedding, $\op{area}(\hat u)\leq\op{area}(M)$ and we find a uniform bound from below
\begin{equation}\label{eq-lb}
 \int_0^1(u,T)^*\eta_k\ +\ S_k(x_0,T)\ \geq\ 0\ +\ 0\ -\ \sup_M|\sigma|\cdot\op{area}(\hat u)\ \geq\ -\sup_M|\sigma|\cdot\op{area}(M)\,.
\end{equation}
This observation gives us the idea of defining a functional on the space of simple contractible loops and look for its global minima. First, we notice that $\int_{B^2}\hat u^*\sigma$ is invariant under an orientation-preserving change of parametrization. In order to make the whole right-hand side of \eqref{int-emb} independent of the parametrization, we ask that $(\gamma,\dot\gamma)\in\Sigma_k$. This implies that
\begin{equation*}
\sqrt{2k}\cdot T\ =\ \ell(x)\,,\quad\quad e(x)\ =\ \ell(x)^2\,.
\end{equation*}
Substituting in \eqref{int-emb}, we get
\begin{equation}
 \int_0^1(u,T)^*\eta_k\ +\ S_k(x_0,T)\ =\ \sqrt{2k}\cdot\ell(\partial D)\ -\ \int_D\sigma\ =:\,\mathcal T_k(D)\,,
\end{equation}
where
\begin{equation*}
D=[\hat u]\in\mathcal D(M):=\left\{\begin{aligned}
\mbox{embeddings } \hat u:B^2\longrightarrow M\,,\hspace{50pt}\\
\mbox{ up to orientation}\mbox{-preserving reparametrizations }\end{aligned}\right\}
\end{equation*}
and $\partial D$ represents the boundary of $D$ oriented in the counter-clockwise sense. 
We readily see that the critical points of this functional correspond to the periodic orbits we are looking for.
\begin{prp}
If $D$ is a critical point of $\mathcal T_k:\mathcal D(M)\rightarrow\R$, then $\partial D$ is the support of a simple contractible periodic orbit with energy $k$.
\end{prp}
In view of this proposition and the fact that $\mathcal T_k$ is bounded from below, we consider a minimizing sequence $(D_m)_{m\in\N}\subset\mathcal D(M)$. However, the sequence $D_m$ might converge to a disc $D_\infty$ which is not embedded. For example, $D_\infty$ might have a self-tangency at some point $q$ on its boundary (see Figure \ref{pic}).
\begin{figure}
\includegraphics[width=3.5in]{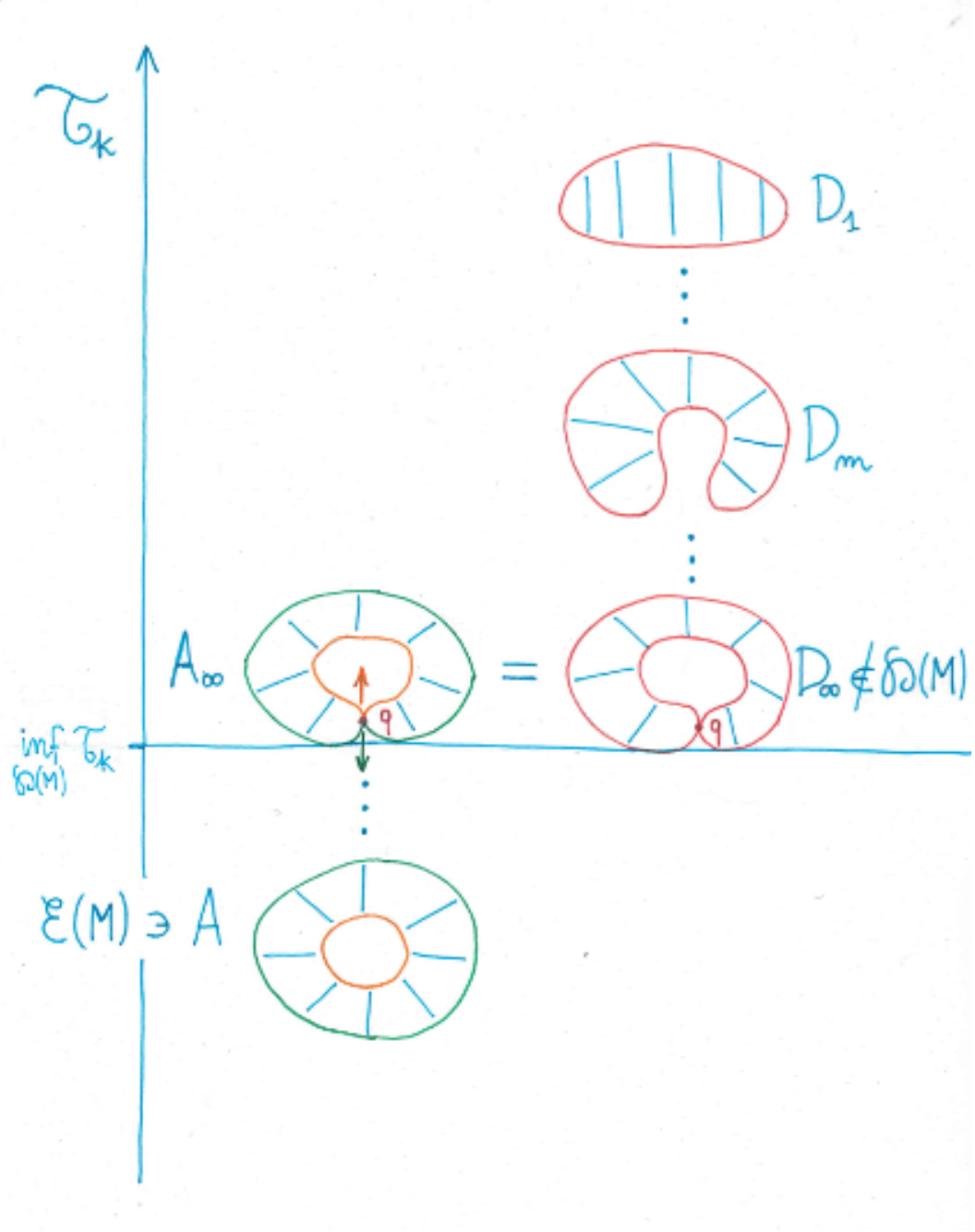}
\caption{Minimizing sequence for $\mathcal T_k$ on $\mathcal D(M)$}\label{pic}
\end{figure} However, in this case the support of $D_\infty$ in $M$ can be interpreted as an annulus $A_\infty$ whose two boundary components touch exactly at $q$. Now we can resolve the singularity in the space of annuli and get an embedded annulus $A$ close to $A_\infty$. The key observation is that $\mathcal T_k$ can be extended to the space of annuli and that
\begin{equation}\label{ine-tai}
\mathcal T_k(D_\infty)\ =\ \mathcal T_k(A_\infty)\ >\ \mathcal T_k(A)\,.
\end{equation}
To justify the inequality in the passage above, we observe that $\ell(\partial A)<\ell(\partial A_\infty)$ from classic estimates in Riemannian geometry and that the contribution given by the integral of $\sigma$ is of higher order. This heuristic argument prompts us to give the following definitions.
\begin{dfn}
Let $\mathcal E(M)=\{\mbox{oriented embedded surfaces }\Pi\rightarrow M\}\cup\{\emptyset\}$ and denote by $\mathcal E_+(M)$ and $\mathcal E_-(M)$ the surfaces having the same orientation as $M$ and the opposite orientation, respectively. If $\Pi\in\mathcal E(M)$, then $\partial \Pi$ denotes the (possibly empty) multi-curve made by the boundary components of $\Pi$. If we define the length $\ell(\partial \Pi)$ as the sum of the lengths of the boundary components, we have a natural extension
\begin{align*}
\mathcal T_k:\mathcal E(M)&\ \longrightarrow\ \R\\
\Pi&\ \longmapsto\ \sqrt{2k}\cdot\ell(\partial \Pi)\ -\ \int_\Pi\sigma\,.
\end{align*}
\end{dfn}
\noindent As in \eqref{eq-lb} we find that $\mathcal T_k$ is bounded from below by $-\sup|\sigma|\cdot\op{area}(M)$. Moreover, we observe that there is a bijection
\begin{equation}\label{tai-inv}
\begin{aligned}
\mathcal E_+(M)&\ \longrightarrow\ \mathcal E_-(M)\\
\Pi&\ \longmapsto\ M\setminus\mathring{\Pi}
\end{aligned}
\quad\quad\quad\mbox{such that}\quad\quad
\begin{aligned}
\mathcal T_k(M\setminus\mathring{\Pi})\ =\ \mathcal T_k(\Pi)\ +\ \int_M\sigma\,.
\end{aligned}
\end{equation}
Therefore, it is enough to look for a minimizer on $\mathcal E_-(M)$. The chain of inequalities \eqref{ine-tai} hints at the following result. 
\begin{prp}\label{prp-tai}
For all $k>0$, there exists a minimizer $\Pi^k$ of $\mathcal T_k|_{\mathcal E_-(M)}$. If $\partial \Pi^k=\{\gamma^k_i\}_i$, then the $\gamma^k_i$ are periodic orbits with energy $k$.
\end{prp}
For a proof of this proposition we refer to \cite{Tai93} and \cite{CMP04}:
\begin{itemize}
 \item In the former reference, Ta{\u\i}manov uses a finite dimensional reduction and works on the space of surfaces $\Pi\in\mathcal E(M)$ whose boundary is made by piecewise solutions of the twisted Euler-Lagrange equations with energy $k$. Such a method was also recently extended to general Tonelli Lagrangians on surfaces in \cite{AM16}.
 \item In the latter reference, the authors use a weak formulation of the problem on the space of integral currents $I_2(M)\supset \mathcal E(M)$.
\end{itemize}
In order to use Proposition \ref{prp-tai} to prove the existence of periodic orbits with energy $k$, we have to ensure that $\partial \Pi^k\neq\emptyset$. To this purpose, we observe that $\partial\Pi^k=\emptyset$ implies $\Pi^k\in\{\emptyset,\overline M\}$, where $\overline M$ is $M$ with the opposite orientation. We easily compute
$\mathcal T_k(\emptyset)=0$ and $\mathcal T_k(\overline M)=\int_M\sigma\geq0$. Therefore, for every $k>0$ we have
\begin{equation*}
\inf_{\mathcal E_-(M)}\mathcal T_k\ \leq\ 0\quad\quad\mbox{and}\quad\quad \Big(\ \inf_{\mathcal E_-(M)}\mathcal T_k\ <\ 0\quad\Longrightarrow\quad\partial \Pi^k\ \neq\ \emptyset\ \Big).
\end{equation*}
Since the family of functionals $\mathcal T_k$ is monotone in $k$, we are led to define
\begin{equation}
\tau(g,\sigma)\,:=\ \inf\Big\{\,k\ \big|\ \inf_{\mathcal E_-(M)}\mathcal T_k\,=\,0\,\Big\}\,.
\end{equation}
\begin{prp}
The value $\tau(g,\sigma)$ is a non-negative real number. Moreover,
\begin{equation*}
\tau(g,\sigma)\ >\ 0\ \ \quad\Longleftrightarrow\quad\ \ \sigma_{q_0}\ <\ 0\,, \mbox{ for some }q_0\in M\,.
\end{equation*}
If $\sigma$ is exact, then
\begin{equation}
\tau(g,\sigma)\ =\ c_0(g,\sigma)\,:=\ \inf_{d\theta=\sigma}\sup_{q\in M}|\theta_q|\,.
\end{equation}
\end{prp}
\noindent We leave the proof of the first statement of the proposition as an exercise to the reader. The second statement follows from \cite{CMP04}. We can summarize our answer to the question raised at the beginning of this section with the following theorem.
\begin{thm}\label{thm-tai}
Suppose that there exists $q_0\in M$ such that $\sigma_{q_0}<0$. Then, we can find a positive real number $\tau(g,\sigma)$, coinciding with $c_0(g,\sigma)$ when $\sigma$ is exact, such that for every $k\in(0,\tau(g,\sigma))$, there exists a non-empty set of simple periodic orbits $\{\gamma^k_i\}$ having energy $k$ and satisfying
\begin{equation*}
\sum_i\ [\gamma^k_i]\ =\ 0\,\in\, H^1(M;\Z)\,.
\end{equation*}
\end{thm}

\section{Magnetic flows on surfaces II: stable energy levels}\label{sec:sta}

In this last section we continue the study of twisted Lagrangian flows of kinetic type on surfaces by investigating the stability properties of their energy levels. To have a better geometric intuition, we are going to pull-back the twisted symplectic form to the tangent bundle. Thus, let $\flat :TM\rightarrow T^*M$ be the duality isomorphism given by $g$. We define the twisted tangent bundle as the symplectic manifold $(TM,\omega_{g,\sigma})$, where $\omega_{g,\sigma}:=d(\flat^*\lambda)-\pi^*\sigma$. We readily see that $X_{(g,\sigma)}$ is the Hamiltonian flow of $E$ with respect to the symplectic form $\omega_{g,\sigma}$. In this language, our problem is to understand when the hypersurface $\Sigma_k$ is stable in the twisted tangent bundle. We will summarize the current knowledge on the subject in the following four propositions.

The first one sheds light on the relation between stability and the contact property in the generic case.
\begin{prp}\label{prp-st1}
Let $k>0$. If $[\sigma]\neq0$ and $M=\T^2$, $\Sigma_k$ is not of contact type. Moreover, if $X_{(g,\sigma)}|_{\Sigma_k}$ does not admit any non-trivial integral of motion, then:
\begin{enumerate}
\item If $[\sigma]=0$ or $M\neq\T^2$ and $[\sigma]\neq0$, $\Sigma_k$ is stable if and only if it is of contact type.
\item If $M=\T^2$ and $[\sigma]\neq0$, every stabilizing form on $\Sigma_k$ is closed and it has non-vanishing integral over the fibers of $\pi$.
\end{enumerate}
\end{prp}
\noindent The second proposition gives obstruction to the contact property.
\begin{prp}\label{prp-st2}
The following statements hold true. 
\begin{enumerate}
 \item If $[\sigma]=0$, then $\Sigma_k$ is not of negative contact type.
 \item If $[\sigma]\neq0$, then
 \begin{enumerate}
  \item if $M=S^2$, $\Sigma_k$ is not of negative contact type;
  \item if $M$ has genus higher than $1$, there exists $c_h(g,\sigma)>0$ such that
  \begin{itemize}
 \item $\Sigma_k$ is not of negative contact type, when $k>c_h(g,\sigma)$;
 \item $\Sigma_{c_h(g,\sigma)}$ is not of contact type;
 \item $\Sigma_k$ is not of positive contact type, when $k<c_h(g,\sigma)$;
\end{itemize}
 \end{enumerate}
\end{enumerate}
\end{prp}
\noindent The third proposition deals with positive results on stability.
\begin{prp}\label{prp-st3}
The following statements hold true.
\begin{enumerate}
 \item If $[\sigma]=0$, $\Sigma_k$ is of contact type if $k>c_0(g,\sigma)$. If $M=\T^2$, for every Riemannian metric $g$ there exists an exact form $\sigma_g$ for which $\Sigma_{c_0(g,\sigma_g)}$ is of contact type.
 \item If $[\sigma]\neq0$ and $M\neq\T^2$, $\Sigma_k$ is of contact type for $k$ big enough.
 \item If $\sigma$ is a symplectic form on $M$, then $\Sigma_k$ is stable for $k$ small enough.
\end{enumerate}
\end{prp}
\noindent The last proposition deals with negative results on stability.
\begin{prp}\label{prp-st4}
The following statements hold true.
\begin{enumerate}
 \item If $[\sigma]=0$ and $M\neq\T^2$, $\Sigma_k$ is not of contact type, for $k<c_0(g,\sigma)$;
 \item If $[\sigma]\neq0$ and there exists $q\in M$ such that $\sigma_{q}<0$, then
 \begin{enumerate}
 \item when $M\neq\T^2$, $\Sigma_k$ is not of contact type, for $k$ low enough;
 \item when $M=\T^2$, $\Sigma_k$ does not admit a closed stabilizing form, for $k$ low enough.
 \end{enumerate}
\item If $M=S^2$, there exists an energy level associated to some $\overline g$ and some everywhere positive form $\overline\sigma$, which is not of contact type. 
\end{enumerate}
\end{prp}
Before embarking in the proof of such propositions, we make the following observation.
\begin{lem}
Let $k>0$ and set $s:=1/\sqrt{2k}$. Then, the flows of $\Phi^{(g,\sigma)}|_{\Sigma_k}$ and $\Phi^{(g,s\sigma)}|_{\Sigma_{1/2}}$ are conjugated up to a time reparametrization.
\end{lem}
\begin{proof}
By Section \ref{sub:hom} we know that the projections to $M$ of the trajectories of $\Phi^{(g,\sigma)}|_{\Sigma_k}$ and of $\Phi^{(g,s\sigma)}|_{\Sigma_{1/2}}$ both satisfy the equation $\kappa_\gamma=s\cdot f(\gamma)$. Therefore, if
\begin{equation*}
t\ \longmapsto\ \left(\gamma(t),\frac{d\gamma}{dt}(t)\right)
\end{equation*}
is a trajectory of the former flow and we set $\gamma_s(t')\,:=\ \gamma(st')$, then
\begin{equation*}
t'\ \longmapsto\ \left(\gamma_s(t'),\frac{d\gamma_s}{dt'}(t')\right)\ =\ \left(\gamma(st'),\,s\cdot\frac{d\gamma}{dt}(st')\right)
\end{equation*}
is a trajectory of the latter flow.
\end{proof}
\noindent Therefore, given $(g,\sigma)$, instead of studying the flow $\Phi^{(g,\sigma)}$ on each energy level $\Sigma_k$, we can study the $1$-parameter family of flows $\Phi^{(g,s\sigma)}$ on $SM:=\Sigma_{1/2}$ as $s$ varies in $(0,+\infty)$. The advantage of rescaling $\sigma$ is that now we can work on a fixed three-dimensional manifold: $SM$. The tangent bundle of $SM$ has a global frame $(X,V,H)$ and corresponding dual co-frame $(\alpha,\psi,\beta)$, which we now define.

Let $\mathcal H\subset SM$ be the horizontal distribution given by the Levi-Civita connection of $g$. For every $(q,v)\in SM$, $X_{(q,v)}$ and $H_{(q,v)}$ are defined as the unique elements in $\mathcal H$ such that
\begin{equation*}
d_{(q,v)}\pi\big(X_{(q,v)}\big)\ =\ v\,,\quad\quad d_{(q,v)}\pi\big(H_{(q,v)}\big)\ =\ \imath \cdot v\,.
\end{equation*}
Analogously, $\alpha_{(q,v)}$ and $\beta_{(q,v)}$ are defined by
\begin{equation*}
\alpha_{(q,v)}(\cdot)\ =\ g_q\big(v,d_{(q,v)}\pi(\cdot)\big)\,,\quad\quad \beta_{(q,v)}(\cdot)\ =\ g_q\big(\imath\cdot v,d_{(q,v)}\pi(\cdot)\big)\,.
\end{equation*}
The vector $V$ is the generator of the rotations along the fibers $\varphi\mapsto(q,\cos\varphi\, v+\sin\varphi\,\imath \cdot v)$.
The form $\psi$ is the connection $1$-form of the Levi-Civita connection. If $W\in T_{(q,v)}SM$ and $w(t)=(\gamma(t),v(t))$ is a curve such that $w(0)=(q,v)$ and $\dot w(0)=W$, then
\begin{equation*}
\psi_{(q,v)}(W)\ =\ g_q\big(\nabla_{\dot\gamma(0)}v,\imath\cdot v\big).
\end{equation*}
Finally, we orient $SM$ using the frame $(X,V,H)$.

The proof of the following proposition giving the structural relations for the co-frame is a particular case of the identities proven in \cite{GK02}.
\begin{prp}
Let $K$ be the Gaussian curvature of $g$. We have the relations:
\begin{equation}
d\alpha\ =\ \psi\wedge\beta\,,\quad\quad d\psi\ =\ K\beta\wedge\alpha\ =\ -K\pi^*\mu\,,\quad\quad d\beta\ =\ \alpha\wedge\psi\,.
\end{equation}
\end{prp}
\noindent Using the frame $(X,V,H)$ we can write
\begin{equation*}
X_s\,:=\ X_{(g,s\sigma)}\ =\ X\,+\,sfV\,,\quad\quad \omega_s\,:=\ \omega_{g,s\sigma}|_{SM}\ =\ d\alpha\,-\,s\pi^*\sigma\,.
\end{equation*}
We also use the notation $\Phi^s$ for the flow of $X_s$ on $SM$.

\subsection{Stability of the homogeneous systems}

Let us start by describing the stability properties of the homogeneous examples introduced in Section \ref{sub:hom}.
\subsubsection{The two-sphere}
In this case we have $\sigma=\mu=K\mu$. Hence,
\begin{equation*}
\omega_s\ =\ d\alpha-s\pi^*\sigma\ =\ d(\alpha+s\psi)\quad \mbox{and}\quad (\alpha+s\psi)(X_s)\ =\ (\alpha+s\psi)(X+sV)\ =\ 1\,+\,s^2\,.
\end{equation*}
Every energy level is of positive contact type.
\subsubsection{The two-torus}
In this case we compute
\begin{equation*}
d\psi\ =\ K\mu\ =\ 0\quad \mbox{and}\quad \psi(X_s)\ =\ \psi(X+sV)\ =\ s\,.
\end{equation*}
Every energy level is stable.
\subsubsection{The hyperbolic surface}
In this case we have $\sigma=\mu=-K\mu$. Hence,
\begin{equation*}
\omega_s\ =\ d\alpha-s\pi^*\sigma\ =\ d(\alpha-s\psi)\quad \mbox{and}\quad (\alpha-s\psi)(X_s)\ =\ (\alpha-s\psi)(X+sV)\ =\ 1\,-\,s^2\,.
\end{equation*}
Every energy level $\Sigma_k$ with $k>\frac{1}{2}$ is of positive contact type. Every energy level $\Sigma_k$ with $k<\frac{1}{2}$ is of negative contact type.
As follows from Proposition \ref{prp-st2}, $c_h(g,\sigma)=1/2$ and $\Sigma_{1/2}$ is not stable.
\bigskip

\subsection{Invariant measures on $SM$}
A fundamental ingredient in the proof of the four propositions is the notion of invariant measure for a flow. In this subsection, we recall this notion and we observe that twisted systems of purely kinetic type always possess a natural invariant measure called the \textit{Liouville measure}.

\begin{dfn}
A Borel measure $\xi$ on $SM$ is $\Phi^{s}\mathsf{-invariant}$, if $\xi(\Phi^{s}_{t}(A))=\xi(A)$, for every $t\in\R$ and every Borel set $A$. This is equivalent to asking
\begin{equation}
\int_{SM}dh(X_s)\,\xi\ =\ 0\,,\quad\quad\forall\,h\in C^\infty(SM,\R)\,.
\end{equation}
The $\mathsf{rotation\ vector}$ of $\xi$ is $\rho(\xi)\in H_1(SM,\R)$ defined by duality on $[\tau]\in H^1(SM,\R)$:
\begin{equation}
<[\tau],\rho(\xi)>\ =\ \int_{SM}\tau(X_s)\,\xi\,,
\end{equation}
where $\tau\in\Omega^1(SM)$ is any closed form representing the class $[\tau]$.
\end{dfn}
Since $X_{s}$ is a section of $\ker\omega_s$ and $\omega_s$ is nowhere vanishing, we can find a unique volume form $\Omega_s$ such that $\imath_{X_{s}}\Omega_s=\omega_s$. We can write $\Omega_s=\tau_s\wedge\omega_s$, where $\tau_s$ is any $1$-form such that $\tau_s(X_{s})=1$. We easily see that $\alpha(X+sfV)=1+0$. Hence, $\Omega_s=\alpha\wedge\omega_s=\alpha\wedge d\alpha$. Notice, indeed, that $\alpha\wedge\pi^*\sigma=0$ since it is annihilated by $V$.
\begin{dfn}
The $\mathsf{Liouville\ measure\ \xi_{SM}}$ on $SM$ is the Borel measure defined by integration with the differential form $\alpha\wedge d\alpha$. It is an invariant measure for $\Phi^{s}$ for every $s>0$.
\end{dfn}
In order to compute the rotation vector of $\xi_{SM}$, we need a lemma which tells us when $\omega_s$ is exact. The easy proof is left to the reader.
\begin{lem}\label{lem-exa}
If $\sigma$ is exact, then $\pi^*\sigma$ is exact and we have an injection
\begin{equation}
\begin{aligned}
\mbox{Primitives of }\ \sigma&\ \xrightarrow{\quad\quad}\ \mbox{Primitives of }\ \omega_s\\
\zeta&\ \xmapsto{\quad\quad}\ \alpha\,-\,s\pi^*\zeta\,.
\end{aligned}
\end{equation}
If $M\neq\T^2$, then $\pi^*\sigma$ is exact and we have an injection
\begin{equation}
\begin{aligned}
\mbox{Primitives of }\ \sigma\,-\,\frac{[\sigma]}{2\pi\chi(M)}K\mu&\ \xrightarrow{\quad\quad}\ \mbox{Primitives of }\ \omega_s\\
\zeta&\ \xmapsto{\quad\quad}\ \alpha\,-\,s\pi^*\zeta\,+\,s\frac{[\sigma]}{2\pi\chi(M)}\psi\,.
\end{aligned}
\end{equation}
If $M=\T^2$ and $\sigma$ is non-exact, then $\omega_s$ is non-exact.
\end{lem}
\noindent We can now state a proposition concerning $\rho(\xi_{SM})$.
\begin{prp}\label{prp-rot}
If $[\sigma]\neq0$ and $M=\T^2$, then there holds $\rho(\xi_{SM})=s[\sigma]\cdot[S_qM]$, where $[S_qM]\in H_1(SM,\Z)$ is the class of a fiber of $SM\rightarrow M$ oriented counter-clockwise. Otherwise, $\rho(\xi_{SM})=0$.
\end{prp}
\begin{proof}
Let $[\tau]\in H^1(SM;\R)$. We notice that
\begin{equation*}
\tau(X_{s})\,\alpha\,\wedge\, d\alpha\ =\ \imath_{X_s}\Big(\tau\,\wedge\,\alpha\,\wedge\, d\alpha\Big)\ +\ \tau\,\wedge\,\imath_{X_s}\big(\alpha\,\wedge\, d\alpha\big)\ =\ 0\ +\ \tau\,\wedge\,\omega_s\,.
\end{equation*}
Therefore,
\begin{equation*}
<[\tau],\rho(\xi_{SM})>\ =\ \int_{SM}\tau\,\wedge\,\omega_s\ =\ s\int_{SM}\tau\,\wedge\,\pi^*\sigma\,.
\end{equation*}
If $M=\T^2$, then $S\T^2\simeq S^1\times \T^2$ and we can use Fubini's theorem to integrate separately in the vertical directions and in the horizontal direction. Observe that since $\tau$ is closed, the integral over a fiber $S_q\T^2$ does not depend on $q$. Thus we find
\begin{equation*}
 \int_{S\T^2}\tau\,\wedge\,\pi^*\sigma\ =\ <[\tau],[S_q\T^2]>\cdot\,[\sigma]\,.
\end{equation*}
and the proposition is proven for the $2$-torus. When $M\neq\T^2$, $\pi^*\sigma$ is exact and, therefore,
$\int_{SM}\tau\wedge\pi^*\sigma=0$. The proposition is proven also in this case.
\end{proof}
\noindent We now proceed to the proofs of the four propositions.
\bigskip

\subsection{Proof of Proposition \ref{prp-st1}}
If $M=\T^2$ and $[\sigma]\neq0$, then $\omega_s$ is not exact by Lemma \ref{lem-exa}. In particular, $SM$ cannot be of contact type. This proves the first statement of the proposition. Now let $\tau_s\in\Omega^1(SM)$ be a stabilizing form for $\omega_s$. Since $\ker(d\tau_s)\supset \ker\omega_s$, there exists a function $\rho_s:SM\rightarrow\R$ such that $d\tau_s=\rho_s\omega_s$. Taking the exterior differential in this equation, we get $0=d\rho_s\wedge\omega_s$. Plugging in the vector field $X_{s}$ we get $0=d\rho_s(X_{s})\omega_s$. Since $\omega_s$ is nowhere zero, we conclude that $d\rho_s(X_{s})=0$. Namely, $\rho_s$ is a first integral for the flow. By assumption, $\rho_s$ is equal to a constant. If $\rho_s=0$, then $\tau_s$ is closed, if $\rho_s\neq0$, then $\tau_s$ is a contact form. Suppose the first alternative holds. Since $\tau_s(X_s)\neq0$ everywhere, we have
\begin{equation*}
0\ \neq\ \int_{SM}\tau_s(X_s)\xi_{SM}\ =\ <[\tau_s],\rho(\xi_{SM})>\,.
\end{equation*}
By Proposition \ref{prp-rot}, this can only happen if $M=\T^2$ and $<[\tau_s],[S_q\T^2]>\neq0$, which is what we had to prove.
\bigskip

\subsection{Proof of Proposition \ref{prp-st2}}
The proof of the second proposition is based on the fact that when $\omega_s$ is exact we can associate a number to every invariant measure with zero rotation vector.
\begin{dfn}
Suppose $\omega_s$ is exact and that $\xi$ is a $\Phi^s$-invariant measure with $\rho(\xi)=0$. We define the $\mathsf{action}$ of $\xi$ as the number
\begin{equation}\label{eq-act}
\mathcal S_s(\xi)\,:=\ \int_{SM}\tau_s(X_s)\,\xi\,,
\end{equation}
where $\tau_s$ is any primitive for $\omega_s$. Such number does not depend on $\tau_s$ since $\rho(\xi)=0$.
\end{dfn}
\noindent The action of invariant measures gives an obstruction to being of contact type.
\begin{lem}\label{lem-ct}
Suppose $\omega_s$ is exact and that $\xi$ is a non-zero $\Phi^s$-invariant measure with $\rho(\xi)=0$. If $\mathcal S_s(\xi)\leq0$, then $SM$ cannot be of positive contact type. If $\mathcal S_s(\xi)\geq0$, then $SM$ cannot be of negative contact type.
\end{lem}
\begin{proof}
If $SM$ is of positive contact type, there exists $\tau_s$ such that $d\tau_s=\omega_s$ and $\tau_s(X_s)>0$. Therefore,
\begin{equation*}
\mathcal S_s(\xi)\ =\ \int_{SM}\tau_s(X_s)\,\xi\ \geq\ \inf_{SM}\tau_s(X_s)\cdot\xi(SM)\ >\ 0\,.
\end{equation*}
For the case of negative contact type, we argue in the same way.
\end{proof}
Let us now compute the action of the Liouville measure.
\begin{prp}\label{prp-act}
If $\sigma$ is exact, then
\begin{equation}\label{act-lio1}
\mathcal S_s(\xi_{SM})\ =\ \xi_{SM}(SM)\ =\ 2\pi[\mu]\,. 
\end{equation}
If $M\neq\T^2$, then
\begin{equation}\label{act-lio2}
\mathcal S_s(\xi_{SM})\ =\ \xi_{SM}(SM)\ +\ s^2\frac{[\sigma]^2}{\chi(M)}\,. 
\end{equation}
\end{prp}
\begin{proof}
If $\sigma=d\zeta$, then $\alpha-s\pi^*\zeta$ is a primitive of $\omega_s$ by Lemma \ref{lem-exa} and we have
\begin{equation}\label{eq-funex}
(\alpha-s\pi^*\zeta)(X_s)_{(q,v)}\ =\ 1\ -\ s\zeta_q(v)\,,\quad\forall\,(q,v)\in SM\,.
\end{equation}
Consider the \textit{flip} $I:SM\rightarrow SM$ given by $I(q,v):=(q,-v)$. We see that
\begin{equation*}
(I^*\alpha)_{(q,v)}\ =\ \alpha_{I(q,v)}dI\ =\ g_q(-v,d\pi dI\cdot)\ =\ \alpha_{I(q,v)}\,.
\end{equation*}
Hence $\xi_{SM}$ is $I$-invariant. However, $\zeta\circ I(q,v)=-\zeta(q,v)$. Therefore,
\begin{equation}\label{eq-ex0}
\int_{SM}\zeta\,\xi_{SM}\ =\ 0
\end{equation}
and from the definition of action given in \eqref{eq-act}, we see that \eqref{act-lio1} is satisfied. To prove the second identity, we consider a primitive $\alpha-s\pi^*\zeta+s\frac{[\sigma]}{2\pi\chi(M)}\psi$ for $\omega_s$ as prescribed by Lemma \ref{lem-exa}. We compute
\begin{equation}
\left(\alpha\ -\ s\pi^*\zeta\ +\ s\frac{[\sigma]}{2\pi\chi(M)}\psi\right)(X_s)_{(q,v)}\ =\ 1\ -\ s\zeta_q(v)\ +\ s^2\frac{[\sigma]}{2\pi\chi(M)}f(q)\,.
\end{equation}
Thus, we need to estimate the integral of $f\circ\pi$ on $SM$. Let $U_i$ be an open cover of $M$ such that $SU_i\simeq S^1\times U_i$ and let $a_i$ be a partition of unity subordinated to it. We have
\begin{align*}
\int_{SM}f(q)\,\alpha\wedge d\alpha\ =\ \int_{SM}f(q)\,\alpha\wedge\psi\wedge\beta\ &=\ -\int_{SM}f(q)\,\psi\wedge\pi^*\mu\\
&=\ -\sum_i\int_{SU_i}a_i(q)\,\psi\wedge\pi^*\sigma\\
&=\ -\sum_i\int_{S^1\times U_i}a_i(q)\,(-d\varphi\wedge\pi^*\sigma)\\
&=\ \sum_i\int_{U_i}a_i(q)\,\sigma\int_{S^1}d\varphi\\
&=\ 2\pi\sum_i\int_{U_i}a_i(q)\,\sigma\\
&=\ 2\pi[\sigma]\,,
\end{align*}
where $\varphi$ is an angular coordinate on $S_qU_i$ going in the clockwise direction (hence the presence of an additional minus sign in the third line). Putting this computation together with \eqref{eq-ex0}, we get the desired identity.
\end{proof}
\noindent Proposition \ref{prp-st2} now follows from Lemma \ref{lem-ct} and Proposition \ref{prp-act} after defining
\begin{equation}
c_h(g,\sigma)\,:=\ -\,\frac{[\sigma]^2}{4\pi\chi(M)[\mu]}\,,\quad\mbox{when $M$ has genus higher than one}\,.
\end{equation}
\begin{rmk}
We have seen in the homogeneous example above that $c_h(g,\sigma)=c(g,\sigma)$. The relation between $c_h$ and the Ma\~n\'e critical value was studied in general by G.\ Paternain in \cite{Pat09}. There the author proves that $c_h(g,\sigma)\leq c(g,\sigma)$ and that $c_h(g,\sigma)=c(g,\sigma)$ if and only if $g$ is a metric of constant curvature and $\sigma$ is a multiple of the area form. 
\end{rmk}

\bigskip

\subsection{Proof of Proposition \ref{prp-st3}}
Suppose that $\sigma$ is exact and let us consider a primitive $\alpha-s\pi^*\zeta$ given by Lemma \ref{lem-exa}. We have
\begin{equation*}
(\alpha\ -\ s\pi^*\zeta)(X_s)_{(q,v)}\ =\ 1\ -\ s\zeta_q(v)\ \geq\ 1\ -\ s\sup_{M}|\zeta|\,.
\end{equation*}
Requiring that the right hand-side is positive is equivalent to saying that
\begin{equation*}
k\ =\ \frac{1}{2s^2}\ >\ \sup_M\frac{1}{2}|\zeta|^2\,.
\end{equation*}
Since this holds for every $\zeta$ which is a primitive for $\sigma$, we have that the last inequality is equivalent to $k>c_0(g,\sigma)$. 
Contreras, Macarini and G.\ Paternain also found in \cite{CMP04} examples of exact systems on $\T^2$, which are of contact type for $k=c_0(g,\sigma)$ (see also \cite[Section 4.1.1]{Ben14}). We will not discuss these examples here and we refer the reader to the cited literature for more details.

Let us now deal with the non-exact case. If $M\neq \T^2$, then we consider a primitive of the form $\alpha-s\pi^*\zeta+s\frac{[\sigma]}{2\pi\chi(M)}\psi$ and we compute
\begin{equation}\label{quantity}
\Big(\alpha\ -\ s\pi^*\zeta\ +\ s\frac{[\sigma]}{2\pi\chi(M)}\psi\Big)(X_s)_{(q,v)}\ =\ 1\ -\ s\zeta_q(v)\ +\ s^2\frac{[\sigma]}{2\pi\chi(M)}f(q)\,.
\end{equation}
We can give the estimate from below 
\begin{equation*}
1\ -\ s\zeta_q(v)\ +\ s^2\frac{[\sigma]}{2\pi\chi(M)}f(q)\ \geq\ 1\ -\ s\sup_M|\zeta|-s^2\left|\frac{[\sigma]}{2\pi\chi(M)}\right|\cdot\sup_M|f|\,
\end{equation*}
and we see that this quantity is strictly positive for $s$ small enough.

Suppose now that $\sigma$ is a symplectic form on $M$. We have three cases.
\begin{enumerate}
 \item If $M=S^2$, then the quantity in \eqref{quantity} is bounded from below by
\begin{equation*}
1-s\sup_M|\zeta|+s^2\frac{[\sigma]}{4\pi}\cdot\inf_Mf\,.
\end{equation*}
Since $[\sigma]>0$, we have that $\inf f>0$ and we see that such quantity is strictly positive for big $s$.
\item If $M$ has genus larger than $1$, then the quantity in \eqref{quantity} is bounded from above by
\begin{equation*}
1\ +\ s\sup_M|\zeta|\ +\ s^2\frac{[\sigma]}{2\pi\chi(M)}\cdot\inf_Mf\,.
\end{equation*}
Since $\chi(M)<0$ and $\inf f>0$, such quantity is strictly negative for big $s$.
\item If $M=\T^2$, then there exists a closed form $\tau\in\Omega^1(S\T^2)$ such that $\tau(V)=1$ (prove such statement as an exercise). Thus, we get
\begin{equation}
\tau(X_s)\ =\ \tau(X)\ +\ sf\ \geq\ \inf_{SM}\tau(X)\ +\ s\inf_M f\,
\end{equation}
and such quantity is positive provided $\inf f>0$ and $s$ is big enough.
\end{enumerate}
\bigskip

\subsection{Proof of Proposition \ref{prp-st4}}
If $\sigma$ is exact and $k<c_0(g,\sigma)$, we can use Theorem \ref{thm-tai} to find an embedded surface $\Pi\subset M$ with non-empty boundary $\partial \Pi=\{\gamma_i\}$ such that $\mathcal T_k(\Pi)<0$ and the $\gamma_i$'s are periodic orbits of $\Phi^s$ (parametrized by arc-length). Let $(\gamma_i,\dot\gamma_i)$ be the corresponding curve on $SM$ and let $\xi_i$ be the associated invariant measure. Define $\xi_{\partial \Pi}:=\sum_i\xi_i$. What is its rotation vector? Call $\pi_*:H_1(SM;\R)\rightarrow H_1(M;\R)$ the map induced by the projection $\pi$ in homology and observe that
\begin{equation}
\pi_*(\rho(\xi_{\partial \Pi}))\ =\ \sum_i\pi_*(\rho(\xi_i))\ =\ \sum_i\,[\gamma_i]\ =\ [\partial \Pi]\ =\ 0\,.
\end{equation}
\begin{exe}
The map $\pi_*$ is an isomorphism if $M\neq\T^2$. 
\end{exe}
\noindent Thus, we conclude that $\rho(\xi_{\partial \Pi})=0$, if $M\neq\T^2$. Let us compute the action in this case. As before, we use a primitive $\alpha-s\pi^*\zeta$:
\begin{equation}\label{act-tai}
\begin{aligned}
\mathcal S_s(\xi_{\partial \Pi}) = \sum_i \int_{SM}\!\!(1-s\zeta_q(v))\xi_i = \sum_i\int_0^{\ell(\gamma_i)}\!\!\!\!\big(1-s\zeta_{\gamma_i}(\dot\gamma_i)\big)\di t &= \sum_i\ell(\gamma_i)-s\int_0^{\ell(\gamma_i)}\!\!\!\!\gamma_i^*\zeta\\
&= \ell(\partial \Pi)-s\int_{\Pi}\sigma\\
&= s\mathcal T_k(\Pi)\,.
\end{aligned}
\end{equation}
By hypothesis the last quantity is negative and Lemma \ref{lem-ct} tells us that $\Sigma_k$ cannot be of positive contact type. Since by Proposition \ref{prp-st2}, $\Sigma_k$ cannot be of negative contact type either, point \textit{(1)} of the proposition is proved.

We now move to prove point \textit{(2a)} with the aid of a little exercise.
\begin{exe}
We prove a generalization of \eqref{act-tai}, when $M\neq\T^2$. Let $\Pi$ be an embedded surface such that $\partial \Pi$ is a union of periodic orbits and let $\xi_{\partial \Pi}$ be the invariant measure constructed as before. Then,
\begin{equation}\label{eq-gen}
\frac{\mathcal S_s(\xi_{\partial \Pi})}{s}\ =\ \mathcal T_k(\Pi)\ +\ \frac{\mathfrak o(\Pi)\chi(\Pi)[\sigma]}{\chi(M)}\,,
\end{equation}
where $\mathfrak o(\Pi)\in\{+1,-1\}$ record the orientation of $\Pi$. To prove such identity one recalls that $\kappa_{\gamma_i}=sf(\gamma_i)$ and then uses the Gauss-Bonnet theorem (taking into account orientations) to express the integral of the geodesic curvature along $\partial \Pi$. What happens if we consider $M\setminus\Pi$? Do the two expressions for $\mathcal S_s(\xi_{\partial \Pi})$ agree? Remember relation \eqref{tai-inv}.
\end{exe}

The problem with formula \eqref{eq-gen} is that Theorem \ref{thm-tai} does not give any information on the Euler characteristic of $\Pi$. To circumvent this problem we need the following result by Ginzburg \cite{Gin87} (see also \cite[Chapter 7]{AB15}).
\begin{prp}\label{prp-gin}
If $\sup f>\varepsilon$ for some $\varepsilon<0$, there exists a constant $C>0$ such that for every small enough $k$ we can find a simple periodic orbit $\gamma^k_+$ supported on $\{f>\varepsilon\}$ and such that $\ell(\gamma^k_+)\leq\sqrt{2k}C$.

If $\inf f<-\varepsilon$, for some $\varepsilon>0$, there exists $C>0$ such that for every small enough $k$, there exists a simple periodic orbit $\gamma^k_-$ supported on $\{f<-\varepsilon\}$ and such that $\ell(\gamma^k_-)\leq\sqrt{2k} C$.
\end{prp}
\noindent If $f$ is negative at some point, by Proposition \ref{prp-gin}, there exists $\gamma^k_-$ with the properties listed above, for $k$ small. In particular, $\gamma^k_-$ bounds a small disc $D^k_-$. Since the geodesic curvature of $\gamma^k_-$ is very negative, such disc lies in $\mathcal E_-(M)$. When $M\neq\T^2$, we use \eqref{eq-gen} and find
\begin{equation*}
\frac{S_s(\xi_{\partial D^k_-}\big)}{s}\ =\ \mathcal T_k(D^k_-)\ -\ \frac{2}{\chi(M)}[\sigma]\,.
\end{equation*}
By the estimate on the length of $\gamma^k_-$ we get that $|\mathcal T_k(D^k_-)|\leq Ck^2$ (see \eqref{psigma}). Therefore, $S_s(\xi_{\partial D^k_-})$ has the opposite sign of $\chi(M)$ for $k$ small enough. Combining Lemma \ref{lem-ct} and Proposition \ref{prp-st2}, point \textit{(2a)} is proven.

Let us deal now with the case of the $2$-torus. Since $[\sigma]>0$, by Proposition \ref{prp-gin} there exists also $\gamma^k_+$ bounding a disc $D_+^k$. Let $\Pi^k=D^k_-\cup D^k_+$. We claim that the measure $\xi_{\partial \Pi^k}$ has zero rotation vector.
\begin{exe}
Prove the claim by showing that $(\gamma^k_+,\dot\gamma^k_+)$ is freely homotopic in $S\T^2$ to $[S_q\T^2]$, namely the class of a fiber with orientation given by $V$. Analogously, prove that $(\gamma^k_-,\dot\gamma^k_-)$ is freely homotopic to a fiber with the opposite orientation.
\end{exe}
\noindent If $\tau_s$ is a closed stabilizing form, we have that the function $\tau_s(X_s)$ is nowhere zero. Therefore,
\begin{equation*}
0\ \neq\ \int_{S\T^2}\tau_s(X_s)\,\xi_{\partial \Pi^k}\ =\ <[\tau_s],\rho(\xi_{\partial\Pi^k})>\ =\ 0\,,
\end{equation*}
which is a contradiction.

We omit the proof of point \textit{(3)}, for which we refer the reader to \cite{Ben14a}.

\section*{Acknowledgements} We would like to express our gratitude to Ezequiel Maderna and Ludovic Rifford for organizing the research school and for the friendly atmosphere they created while we stayed in Uruguay. We also sincerely thank Marco Mazzucchelli and Alfonso Sorrentino for many engaging discussions during our time at the school.   
\bibliographystyle{amsalpha}
\bibliography{school}
\end{document}